\documentclass[a4paper,12pt,reqno]{amsart}
\usepackage{a4wide,amsmath,amsfonts,amssymb,amsthm,mathrsfs,
setspace,bm,amsxtra,mathtools,mathrsfs,setspace,wasysym,textcomp}
\usepackage{latexsym}
\usepackage{verbatim}
\usepackage{amsbsy}
\usepackage[all,cmtip]{xy}
\usepackage[english]{babel}
\usepackage[utf8]{inputenc}
\xyoption{poly}
\xyoption{curve}
\newtheorem{thm}{Theorem}[section]
\newtheorem{prop}[thm]{Proposition}
\newtheorem{lemma}[thm]{Lemma}
\newtheorem{corol}[thm]{Corollary}
\newtheorem{definition}[thm]{Definition}
\theoremstyle{definition}
\newtheorem{exa}[thm]{Example}
\newtheorem{exas}[thm]{Examples}
\newtheorem{rmrk}[thm]{Remark}
\newtheorem*{thmnn}{\bf Theorem}
\newtheorem*{defnn}{\bf Definition}
\newtheorem{convention}[thm]{\bf Convention}
\newenvironment{ex}{\begin{exa}\rm}{\parbox{2mm}{\hfill}\hfill $\triangle$\end{exa}}

\newenvironment{rmark}{\begin{rmrk}\rm}{\parbox{2mm}{\hfill}\hfill $\triangle$\end{rmrk}}
\numberwithin{equation}{section}
\newcommand{\N}{{\mathbb N}}
\newcommand{\Z}{{\mathbb Z}}
\newcommand{\T}{{\mathbb T}}
\newcommand{\Cx}{{\mathbb C}}
\newcommand{\Oc}{{\mathcal O}}
\newcommand{\fun}{{{\mathbb F}_1}}

\newcommand{\funn}{{{\mathbb F}_{1^n}}}
\newcommand{\adj}{\otimes_{\fun}\N}
\newcommand{\adjz}{\otimes_{\fun}\Z}
\newcommand{\fq}{{\mathbb F}_q}
\newcommand{\spec}{\operatorname{Spec}}
\newcommand{\Hom}{\operatorname{Hom}}
\newcommand{\ihom}{\underline{\Hom}}
\newcommand{\dirlim}{\underrightarrow{\mathrm{lim}}}
\newcommand{\invlim}{\underleftarrow{\mathrm{lim}}}
\newcommand\Ccat{{\mathsf C}}
\newcommand\Dcat{{\mathsf D}}
\newcommand\Fcat{{\mathsf F}}
\newcommand\Bcat{{\mathsf B}}
\newcommand\Gcat{{\mathsf G}}
\newcommand\Set{{\mathsf{Set}}}
\newcommand\Mon{{\mathsf{Mon}}}
\newcommand\SRing{{\mathsf{SRing}}}
\newcommand\Ring{{\mathsf{Ring}}}
\newcommand\Ab{{\mathsf{Ab}}}
\newcommand\mon{{\mathsf{CMon}}}
\newcommand\Mod{{\mathsf{Mod}}}
\newcommand\aff{{\mathsf{Aff}}}
\newcommand\sch{{\mathsf{Sch}}}
\newcommand\sh{{\mathsf{Sh}}}
\newcommand\psh{{\mathsf{Presh}}}

\newcommand\bluep{{\mathsf {Blp}}}
\newcommand\nbluep{{\widetilde{\bluep}}}
\newcommand\nbcat{{\widetilde{\Bcat}}}
\newcommand\trait{\,\text{-}\,}

\newcommand\cF{{\mathcal F}}
\newcommand\cG{{\mathcal G}}

\usepackage{hyperref}
\hypersetup{
    colorlinks=true, 
    linktoc=all,     
    allcolors=blue,  
}
\title[Blueprints and $\fun$-schemes]{\bf Some remarks on blueprints and $\fun$-schemes}

\begin{document}
\thispagestyle{empty}
\date{last revised \today}
\thanks{\\C.~Bartocci was partially supported by {\sc prin}   ``Geometria delle variet\`a  algebriche'', by {\sc gnsaga-in}d{\sc am} and by the University of Genova through the research grant ``Aspetti matematici nello studio delle interazioni fondamentali''.}

\maketitle
\begin{center}{\sc \small Claudio Bartocci,$^\P, ^\S$ Andrea Gentili,$^\P$ and Jean-Jacques Szczeciniarz$^\S$}
\\[10pt]  \small 
$^\P$Dipartimento di Matematica, Universit\`a di Genova,  Genova, Italy \\[3pt]
  $^{\S}${Laboratoire SPHERE, CNRS, Universit\'e Paris Diderot (Paris 7), 75013 Paris, France}
\end{center}
\markright{\sc  C.~Bartocci, A.~Gentili, and J.-J.~Szczeciniarz}
\markleft{\sc Some remarks on blueprints and $\fun$-schemes}
\bigskip\bigskip
\begin{abstract} 
Over the past two decades several different approaches to defining a geometry over $\fun$ have been proposed. In this paper, relying on To\"en and Vaqui\'e's formalism \cite{TV}, we investigate a new category $\sch_\nbcat$ of schemes admitting a Zariski cover by affine schemes relative to the category of blueprints introduced by Lorscheid \cite{Lor12}. A blueprint, which may be thought of as a pair consisting of a monoid $M$ and a relation on the semiring $M\otimes_{\fun} \N$, is a monoid object in a certain symmetric monoidal category $\Bcat$, which is shown to be complete, cocomplete, and closed.  We prove that every  $\nbcat$-scheme $\Sigma$ can be associated, through  adjunctions, with both a classical scheme 
$\Sigma_\Z$ and a scheme $\underline{\Sigma}$ over $\fun$ in the sense of Deitmar  \cite{Dei}, together with a natural transformation $\Lambda\colon \Sigma_\Z\to \underline{\Sigma}\adjz$.  Furthermore, as an application, we show that  the category of ``$\fun$-schemes" defined by A.~Connes and C.~Consani in \cite{CC} can be naturally merged with that of $\nbcat$-schemes to obtain a larger category, whose objects we call ``$\fun$-schemes with relations''.
\end{abstract}

\maketitle

{\small \setcounter{tocdepth}{2}
 \tableofcontents}

\section{Introduction}
\subsection{A quick overview of $\fun$-geometry}
The nonexistent field $\fun$ made its first appearance in Jacques Tits's 1956 paper {\it Sur les analogues alg\'ebriques des groupes semi-simples complexes} \cite{Tits56}.\!\footnote{For a more detailed and exhaustive account of the development of $\fun$-geometry we refer to \cite{LeBruyn11} and \cite{Lor15}.} According to Tits, it was natural to call ``$n$-dimensional projective space over $\fun$'' a set of $n+1$ points, on which the symmetric group $\Sigma_{n+1}$ acts as the group of projective transformations. So, $\Sigma_{n+1}$ was thought of as the group of $\fun$-points of $SL_{n+1}$, and more generally it was conjectured that, for each algebraic group $G$,  one ought to have $W(G)= G(\fun)$, where $W(G)$ is the Weyl group of $G$.

A further strong motivation to seek for  a geometry over $\fun$ was the hope, based on the  multifarious analogies between number fields and function fields, to find some pathway to attack Riemann's hypothesis by mimicking Andr\'e Weil's celebrated proof. The idea behind that, as  explicitly stated in Yuri Manin's influential 1991--92 lectures \cite{ManL}   
and in Kapranov and Smirnov's unpublished paper \cite{KS}, was to regard $\spec \Z$, the final object of the category of schemes, as an arithmetic curve over the ``absolute point'' $\spec \fun$.  Manin's work drew inspiration from Kurokawa's paper \cite{Kur} together with Deninger's results about ``representations of zeta functions as regularized infinite determinants \cite{Den1, Den2, Den3} of certain `absolute Frobenius operators' acting upon a new cohomology theory".
Developing these insights, Manin suggested  a conjectural decomposition of the classical complete Riemann zeta function of the form \cite[eq.~(1.5)]{ManL}
\begin{align}\label{Maninequation}
Z\bigl(\overline{\spec \Z}, s\bigr) &: =  2^{-1/2} \pi^{-s/2} \Gamma(\frac{s}{2}) \zeta(s) = \frac{
\prod^{\tiny\hbox{reg}}_\rho  \frac{s-\rho}{2\pi}}{ {\frac{s}{2\pi}}\frac{s-1}{2\pi}}=\nonumber\\
 &\stackrel{?}{=} \frac {\hbox{det}^{\tiny\hbox{reg}} \bigl( \frac{1}{2\pi} (s\cdot\operatorname{Id} - \Phi ) \big\vert H^1_{?} (\overline{\spec \Z})\bigr)}
 {\hbox{det}^{\tiny\hbox{reg}} \bigl( \frac{1}{2\pi} (s\cdot\operatorname{Id} - \Phi ) \big\vert H^0_{?} (\overline{\spec \Z})\bigr)
 \hbox{det}^{\tiny\hbox{reg}} \bigl( \frac{1}{2\pi} (s\cdot\operatorname{Id} - \Phi ) \big\vert H^2_{?} (\overline{\spec \Z})\bigr)
 }\,,
\end{align}
where the notation $\prod^{\tiny\hbox{reg}}_\rho$ and $\hbox{det}^{\tiny\hbox{reg}}$  refers to ``zeta regularization'' of infinite products and the last hypothetical equality ``postulates the existence of a new cohomology theory $ H^{\bullet}_{?}$, endowed with a canonical `absolute Frobenius' endomorphism $\Phi$''. He conjectured, moreover, that the functions of the form  $\frac{s-\rho}{2\pi}$ in eq.~\ref{Maninequation} could be interpreted as zeta functions according to the definition
$$ Z(\T^\rho, s) = \frac{s-\rho}{2\pi}\,, \quad \rho\geq 0\,,$$
where ``Tate's absolute motive'' $\T$ was to be ``imagined as a motive of a one-dimensional affine line over the absolute point,
$\T^0 = \bullet = \spec \fun$''. 

The first full-fledged definition of variety over ``the field with one element'' was proposed by Christophe Soul\'e in the 1999 preprint \cite{Soule99}; five years later such definition was slightly modified by the same author in the paper \cite{Soule04}).  Taking as a starting point Kapranov and Smirnov's suggestion that $\fun$ should have an extension $\funn$ of degree $n$, Soul\'e insightfully posited that 
$$ \funn \otimes_{\fun} \Z= \Z[T] / (T^n -1) =: R_n\,.$$
Let $\mathsf R$ be the full subcategory of the category $\Ring$ of commutative rings generated by  the rings $R_n$, $n\geq 1$ and their finite tensor products. An affine variety $X$ over $\fun$ is then defined as a covariant functor $\mathsf R \to \Set$ plus some extra data such that there exists a unique (up to isomorphism) affine variety $X_\Z= X\otimes_\fun \Z$ over $Z$ along with an immersion $X \hookrightarrow X_\Z$ satisfying a suitable universal property \cite[D\'efinition 3]{Soule04}. In particular, one has a natural inclusion $X(\funn) \subset (X\otimes_\fun \Z)(R_n)$ for each $n\geq 1$. A notable result proven by Soul\'e was that smooth toric varieties can always be defined over $\fun$.

To formalize $\fun$-geometry Anton Deitmar adopted, in 2005, a different approach, which can be dubbed  as  ``minimalistic'' (using the evocative terminology introduced by Manin in \cite{Man10}). In his terse paper \cite{Dei}, Deitmar associates to each commutative monoid $M$ its ``spectrum over $\fun$'' $\spec M$ consisting of all prime ideals of $M$, i.e.~of all submonoids $P\subset M$ such that $xy \in P$ implies $x\in P$ or $y\in P$. The set $\spec M$ can be endowed with a topology and with a structure (pre)sheaf  $\mathcal O_M$ via localization, just as in the usual case of commutative rings. A topological space $X$ with a sheaf  $\mathcal O_X$ of monoids is then called a ``scheme over $\fun$'', if for every point $x\in X$ there is an open neighborhood $U\subset X$ such that $(U, {\mathcal O}_X\vert_U)$ is isomorphic to $(\spec M, \mathcal O_M)$ for some monoid $M$. The forgetuful functor $\Ring \to \Mon$ has a left adjoint given by $M\mapsto M\otimes_\fun \Z$ (in Deitmar's notation), and this functor extend to a functor  $\trait \otimes_\fun \Z$ from the category of schemes over $\fun$ to the category of classical schemes over $\Z$. Tit's 1957 conjecture stating that $GL_n(\fun)=\Sigma_n$ can be easily proven in Deitmar's theory. Indeed, since $\fun$-modules are just sets and $\funn\otimes_\fun \Z$ has to be isomorphic $\Z^n$, it turns out that $\funn$ can be identified with the set $\{1,\dots, n\}$ of $n$ elements. Hence
$$GL_n(\fun) = \operatorname{Aut}_\fun (\funn) = \operatorname{Aut}(1,\dots,n) = \Sigma_n\,.$$
It is not hard to show, moreover, that the functor $GL_n$ on rings over $\fun$ is represented by a scheme over $\fun$ \cite[Prop.~5.2]{Dei}. As for zeta functions, Deitmar defines, for a scheme $X$ over $\fun$ and for a prime $p$, the formal power series
$$Z_X(p, T) = \exp \bigl( \sum_{n=1}^\infty \frac{T^n}{n}\# X({\mathbb F}_{p^n}) \bigr)\,,    
$$
where ${\mathbb F}_{p^n}$ stands for the field of $p^n$ elements with only its monoidal multiplicative structure and 
$ X({\mathbb F}_{p^n})$ denotes the set of ${\mathbb F}_{p^n}$-valued points of $X$, and proves that $Z_X(p, T)$
coincides with the Hasse--Weil zeta function of $X\otimes_\fun {\mathbb F}_{p^n}$ \cite[Prop.~6.3]{Dei}. Albeit elegant, this result is a bit of a letdown, for --- as the author himself is ready to admit --- it is clear that ``this type of zeta function [...] does not give new insights''.

A natural and extremely general formalism for $\fun$-geometry was elaborated by Bertrand To\"en and Michel Vaqui\'e in their 2009 paper \cite{TV}, tellingly entitled {\em Au dessous de $\spec \Z$}, whose approach appears to be largely inspired by Monique Hakim's work \cite{Hak}.  The authors there showed  how to construct an ``algebraic geometry" relative to any  symmetric monoidal category $\Ccat =(\Ccat,\otimes, \mathbf{1})$, which is supposed to be complete, cocomplete and to admit internal homs. The basic idea is that  the category $\mon_\Ccat$ of commutative (associative and unitary) monoid objects in $\Ccat$ can be taken as a substitute for the category of commutative rings (the monoid objects in the category $\Ab = \Z\trait \Mod$ of Abelian groups) to the end of defining a suitable notion of ``scheme over $\Ccat$''. Each object $V$ of $\mon_\Ccat$ gives rise to the category $V\trait\Mod$ of $V$-modules and each morphism $V \to W$ in  $\mon_\Ccat$ determines a change of basis functor 
$\trait \otimes_V W \colon V\trait\Mod \to W\trait\Mod$;  the category of commutative $V$-algebras can be realized as the category of commutative monoids in $V\trait\Mod$ and is naturally equivalent to the category $V/{\mon_\Ccat}$. An affine scheme over $\Ccat$ is, by definition, an object of the opposite category  $\aff_\Ccat = \mon^{\text{op}}_\Ccat$ and the tautological contravariant functor $\mon_\Ccat \to \aff_\Ccat$ is called $\spec (\trait)$. By means of the  pseudo-functor $M$ that maps an object $V$ in $\mon_\Ccat$ to the category of $V$-modules and a morphism $\spec V \to \spec W$ to the functor $\trait \otimes_V W\colon V\trait\Mod\to W\trait\Mod$,  one may introduce the notions of ``Zariski cover" and ``flat cover''
(``$M$-faithfully flat in To\"en and Vaqui\'e's terminology; see Definition \ref{definitionsofcovers} and Remark \ref{mainremarkonTV} below) and use such notions to equip $\aff_\Ccat$ with two distinct  Grothendieck topologies, called, respectively, the flat and the Zariski topology. These topologies determine two categories of sheave on $\aff_\Ccat$, namely
$\sh^{\text{flat}}(\aff_\Ccat)\subset \sh^{\text{Zar}}(\aff_\Ccat) \subset \psh (\aff_\Ccat)$. At this point, mimicking what is done in classical algebraic geometry, a ``scheme over $\Ccat$'' is defined as a sheaf  in $\sh^{\text{Zar}}(\aff_\Ccat)$ that admits an affine Zariski cover (see Definition ~\ref{definitionsofcoversforsheaves} and  Definition \ref{defschemeoverC} below).
If we take as $\Ccat$ the category $\Set$ of sets endowed with the monoidal structure induced by the Cartesian product,
then the category $\aff_\Set$ is nothing but the category $\Mon^\text{op}$ and the objets of the category $\sch_\Set$ can be thought of --- as remarked by To\"en and Vaqui\'e --- as ``schemes over $\fun$''.  Actually, as proven by Alberto Vezzani in \cite{Vezz}, such schemes, that we shall call {\em monoidal schemes}, turn out to be equivalent to Deitmar's schemes.

Deitmar's schemes appear therefore to constitute the very core of  $\fun$-geometry, not just because their definition is rooted in the basic notion of prime spectrum of a monoid, but especially because they naturally fit into the categorical framework established by  To\"en and Vaqui\'e in \cite{TV}, which admits generalizations in many directions (e.g.~towards a derived algebraic geometry over $\fun$). Nonetheless, they are affected by some intrinsic limitations, which are clearly revealed by a result proven by Deitmar himself in 2008 \cite[Thm.~4.1]{Dei08}:
\begin{thmnn}  {\em Let $X$ be a connected integral $\fun$-scheme of finite type.\!\footnote{A Deitmar's $\fun$-scheme X is said to be of finite type, if it has a finite covering by affine schemes $U_i = \spec M_i$ such that each $M_i$ is a finitely generated monoid. Deitmar proved in \cite{Dei06} that an $\fun$-scheme $X$ is of finite type if and only if $X_\Z$ is a $\Z$-scheme of finite type.} Then every irreducible component of $X_{\Cx}= X_\Z \otimes_\Z \Cx$ is a toric variety. The components of $X_{\Cx}$ are mutually isomorphic as toric varieties.}
\end{thmnn}
Since every toric variety is the lift $X_{\Cx}$ of an $\fun$-scheme $X$, the previous theorem entails that integral $\fun$-schemes of finite type are essentially the same as toric varieties. Now, semisimple  algebraic groups are not  toric varieties, so it is apparent that Deitmar's $\fun$-schemes are too little flexible to implement Tits's conjectural program. 

A possible generalization of Deitmar's geometry over $\fun$ was proposed by Olivier Lorscheid, who introduced the notions of  ``blueprint'' and ``blue scheme'' \cite{Lor12}. The basic idea can be illustrated through the following example. The affine group scheme $(SL_{2})_ \Z$ over the integers is defined as
$$(SL_{2})_ \Z = \spec\bigl( \Z[T_1, T_2, T_3, T_4] / (T_1 T_4 - T_2 T_3 -1) \bigr)\,.
$$
As the relation $T_1 T_4 - T_2 T_3 =1$ does not make sense in the monoid  $\fun[T_1, T_2, T_3, T_4]$, any  naive attempt to adapt the previous definition to get a scheme over $\fun$  will necessarily be unsuccessful. The notion of ``blueprint'' just serves serves the purpose of getting rid of this difficulty:
\begin{defnn} {\em A {\em blueprint} is a pair $B=(R, A)$, where $R$ is a semiring  and $A$ is a multiplicative subset  of $R$ containing $0$ and $1$ and generating $R$ as a semiring. A blueprint morphism $f \colon B_1=(R_1, A_1) \to B_2=(R_2, A_2)$
is a semiring morphism $f\colon R_1 \to R_2$ such that $f (A_1) \subset A_2$.}
\end{defnn}
\noindent
The rationale behind this definition can be explained by considering the following situation: if one is given a monoid $A$ and some relation which does not makes sense in $A$ but becomes meaningful in the semiring $A\otimes_\fun \N$, then one can look at the blueprint $(A\otimes_\fun \N, A)$.   

In the same vein as Deitmar's approach, Lorscheid \cite{Lor12} associates to each blueprint $B$ its spectrum $\spec B$, which turns out to be a locally blueprinted space (i.e.~a topological space endowed with a sheaf of blueprints, such that 
all stalks have a unique maximal ideal).  An affine blue scheme is then defined as a locally blueprinted space that is isomorphic to the spectrum of a blueprint, and a blue scheme as a locally blueprinted space that has a covering by affine blue schemes. Deitmar's schemes over $\fun$ and classical schemes over $\Z$ are recovered as special cases of this definition.

\subsection{About the present paper}
A natural question arises: do blue schemes fit into To\"en and Vaqui\'e's framework? This problem was addressed by Lorscheid himself in his 2017 paper \cite{Lor16} and answered in the negative. Nonetheless, it is possible --- as already pointed out in \cite{Lor16} --- to define a category of schemes (here called {\it $\Bcat$-schemes}) relative (in To\"en and Vaqui\'e's sense) to the category of blueprints. Our first aim is to study these schemes by introducing the category of blueprint {\it in a  purely functorial way}, as the category of monoid objects in a closed, complete and cocomplete symmetric monoidal category 
$\Bcat$.

There is a natural adjunction  $\rho \dashv \sigma \colon \aff_{\Bcat} \to \aff_{\Set_\ast}$ between the category of affine $\Bcat$-schemes and that of affine monoidal schemes. However, since the functor $\rho$ is not continuous w.r.t.~the Zariski topology, this adjunction does not give rise to a geometric morphism between the corresponding category of schemes. This hurdle
may be sidestepped by introducing a larger category $\nbcat$ containing $\Bcat$ and by considering the category of those schemes in 
$\sch_{\nbcat}$ that admit a Zariski cover by affine $\Bcat$-schemes. Such schemes, by a slight abuse of language, will be called
{\it $\nbcat$-schemes}. It will be proved that the  adjunction $\rho \dashv \sigma$ above induce an adjunction
$\widehat\rho \dashv \widehat\sigma$ between the category of $\nbcat$-schemes and that of affine monoidal schemes.
Moreover, it will be shown that every $\nbcat$-scheme $\Sigma$ generates a pair $(\underline\Sigma, \Sigma_\Z)$, where $\underline\Sigma$ is a monoidal scheme and $\Sigma_\Z$ a classical scheme, together with a natural transformation 
$\Lambda\colon \Sigma_\Z\to \underline{\Sigma}\adjz$.

More in detail the present paper is organized as follows.

After briefly recalling in \S\  \ref{SectionTV} the fundamental notions of ``relative algebraic geometry'' and fixing 
our notation,\footnote{This overview is complemented by Appendix A, where we review some basic facts about fibered categories, pseudo-functors, and stacks.}
in \S\ \ref{SectionBschemes} we define the full subcategory $\Bcat$  of the category $\N[\trait]/\Mon_0$ (where the functor $\N[\trait]\colon \Set_\ast \to \Mon_0$ is left adjoint to the forgetful functor $\vert\trait\vert$ from the category $\Set_\ast$ of pointed sets to the category of monoids with ``absorbent object''; see \S\ \ref{Sectionnotation}), whose objects
$(X, \N[X] \to M)$ satisfy the conditions:
\begin{equation*}
\begin{array}{l}
\text{a) the morphism $\N[X] \to M$ is an epimorphism;}\\
\text{b) the composition\ }  
X\to \vert \N[X] \vert \to \vert M\vert\  \text{is a monomorphism.}
\end{array}
\end{equation*}
As proven in Theorem \ref{B-category}, the category $\Bcat$ --- which  corresponds to the category of pointed set endowed with a pre-addition structure introduced  in \cite[\S 4]{Lor16} --- carries a natural structure of symmetric monoidal category. Moreover, this structure is closed, complete, and cocomplete. So, the category $\Bcat$ possesses all the properties necessary to carry out To\"en and Vaqui\'e's program. 

It is quite straightforward to show (Proposition \ref{Bluep-category}) that the category $\bluep$ of monoid objects in $\Bcat$ coincides with the category of blueprints (this result was already stated, in equivalent terms, in \cite[Lemma 4.1]{Lor16}, but we provide a detailed and completely functorial proof). Thus, by applying To\"en and Vaqui\'e's formalism to the category  
$\Bcat$, we define {\em the category $\aff_\Bcat =  \bluep^{\text{op}}$ of affine $\Bcat$-schemes} and then  {\em the category $\sch_\Bcat$ of $\Bcat$-schemes}.

The core of our paper is \S\ \ref{sectionadjunctions}. The natural adjunction between the category $\Mon_0$ and the category $\Set_\ast$ gives rise 
to an adjunction 
$\xymatrix{
\aff_{\Mon_0}  \ar[r]_{\vert\trait\vert}  & \aff_{\Set_\ast}   \ar@/_1.1pc/[l]_{\trait\adj} } 
$ that factorizes as shown in the following diagram  
\begin{equation}\label{adjunctiondiagram0}
\xymatrix{
\aff_{\Mon_0}  \ar[r]^{\vert\trait\vert} \ar[d]^{G} & \aff_{\Set_\ast}   \ar@/_1.5pc/[l]_{\trait\adj} \ar[dl]^{\sigma} \\
\aff_{\Bcat} \ar@<1ex>[ur]^{\rho} \ar@/^0.8pc/[u]^{F}
}
\end{equation}
In Proposition \ref{Bschemesadjunctions} it is  proven that the functor $F$ in the diagram \ref{adjunctiondiagram0} is continuous w.r.t.~the Zariski topology and that the induced functor $\widehat{F}\colon \sh(\aff_{\Bcat}) \to \sh({\aff_{\Mon_0}})$ determines a functor $\widehat{F}\colon \sch_{\Bcat} \to \sch_{\Mon_0}$ between the category of 
$\Bcat$-schemes and that of semiring schemes. Similarly, in 
Proposition \ref{Bschemesadjunctions2} it is shown that the functor $\sigma \colon \aff_{\Set_\ast} \to \aff_{\Bcat}$ in the diagram \ref{adjunctiondiagram0} is continuous w.r.t.~the Zariski topology and that the induced functor $\widehat{\sigma} \colon \sh(\aff_{\Set_\ast}) \to \sh(\aff_{\Bcat})$ determines a functor  $\widehat{\sigma}\colon \sch_{\Set_\ast} \to \sch_{\Bcat}$
between the category of monoidal schemes and that of $\Bcat$-schemes.  One would like the functor $\widehat\sigma$ to have a left adjoint determined by the functor $\rho\colon   \aff_\Bcat \to \aff_{\Set_\ast}$ (see diagram \ref{adjunctiondiagram0}). However, the functor $\rho$, although it preserves Zariski covers, does not commute with finite limits. This difficulty may be overcome by introducing the categories $\nbcat$ and $\nbluep$ containing, respectively, $\Bcat$ and $\bluep$ (Definition \ref{def-nbluep}), and by defining the category $\widetilde{\sch_{\nbcat}}$ of {\it $\nbcat$-schemes} as the subcategory of $\sch_{\nbcat}$ whose objects admit a Zariski cover by affine schemes in $\aff_\Bcat$ (Definition \ref{def-nbschemes}). So, a $\nbcat$-scheme is locally described by blueprints. In this way, one shows  (Theorem \ref{rhosigmaadjunction}) that  there is a geometric morphism 
$$\widehat\rho \dashv \widehat\sigma\colon \widetilde{\sch_{\nbcat}} \to \sch_{\Set_\ast}\,.$$

It follows (see Definition \ref{defBscheme} and the ensuing remarks) that each $\nbcat$-scheme $\Sigma$ determines the following geometric data: 
\begin{itemize}
\item a monoidal scheme $\underline{\Sigma}= \widehat{\rho}(\Sigma)$;
\item a scheme $\Sigma_\Z = \widehat{F}_\Z(\Sigma) $ over $\Z$; 
\item a natural transformation 
$\Lambda\colon \Sigma_\Z\to\underline{\Sigma}\circ |\trait |\cong\underline{\Sigma}\adjz$.
\end{itemize}

In \S \ \ref{SectionFinal2}, as an application of our approach, we investigate the relationship of 
$\nbcat$-schemes and  $\fun$-schemes in the sense of Alain Connes and Caterina Consani \cite{CC}. According to their definition \cite[Def.~4.7]{CC}, 
an $\fun$-scheme is a triple $(\underline{\Xi}, \Xi_\Z, \Phi)$, where
$\underline{\Xi}$ is a monoidal scheme,  $\Xi_\Z$ is a scheme over $\Z$, and 
$\Phi$ is natural transformation $\underline{\Xi}\to \Xi_\Z\circ (\trait\adjz)$, such that the induced natural transformation 
$\underline{\Xi}\circ \vert\trait\vert \to \Xi_\Z$, when evaluated on fields, gives isomorphisms (of sets).
Thus, the category of $\nbcat$-schemes and that of $\fun$-schemes can be combined into a larger category, namely their fibered product over the category of monoidal schemes, whose objects will be called {\em $\fun$-schemes with relations}
(Definition \ref{definitionF1schemewr}). In more explicit terms, a $\nbcat$-scheme
$\Sigma$ determining the pair $(\underline\Sigma, \Sigma_\Z)$ and an $\fun$-scheme 
$(\underline{\Sigma}, \Sigma_\Z', \Phi)$ will give rise to
 a ${\fun}$-scheme with relations denoted by the quadruple $(\underline{\Sigma}, \Sigma_\Z, \Sigma_\Z', \Phi)$.
 The main motivation behind this notion is to combine in a single geometric object both the advantages of blueprint approach 
and the benefits of Connnes and Consani's definition (cf.~Remark \ref{finalremark} for a better explanation).  
Each ${\fun}$-scheme with relations  $(\underline{\Sigma}, \Sigma_\Z, \Sigma_\Z', \Phi)$ (with a slight modification of our terminology, see Convention \ref{convention}) determines a natural transformation 
$$\Psi_1 \colon \Sigma_\Z  \to \Sigma_\Z'$$ 
and a  natural transformation 
\begin{equation*} \Psi_2\colon \Sigma'_{\Bcat}\to \Sigma'_\Z\,,
\end{equation*} 
where $\Sigma'_{\Bcat}$ is a certain pullback sheaf on the category $\Ring$ (defined by the diagram \ref{diagramdefiningSigma'Bcat}). This implies that, given a $\nbcat$-scheme $\Sigma$ underlying a ${\fun}$-scheme with relations, we can think of its ``${\mathbb F}_{1^{q-1}}$-points'' in two different senses, and therefore count them in two different ways, as stated in Proposition \ref{propfirsttransferringmap} and in Theorem \ref{thmsecondftransferringmap}.
An interesting case is  when the $\funn$-points of the underlying monoidal scheme $\underline\Sigma$ are counted 
by a polynomial in $n$. Theorem 4.10 of  \cite{CC} shows that, if $(\underline{\Sigma}, \Sigma_\Z', \Phi)$ is an $\fun$-scheme such that the monoidal scheme $\underline{\Sigma}$ is noetherian and torsion-free, then  $\#\underline\Sigma(\funn) = P(\underline\Sigma, n)$, where 
\begin{equation*} 
P(\underline\Sigma, n) = 
\sum_{x\in  {\underline\Sigma}} \#\Hom (\Oc_{\underline\Sigma, x}^\times, \funn)\,.
\end{equation*}
For an $\fun$-scheme with relations $(\underline{\Sigma}, \Sigma_\Z, \Sigma_\Z', \Phi)$    such that
the underlying $\nbcat$-scheme $\Sigma$ is noetherian and torsion-free (Definition \ref{definitionBschemenoethtorsionfree}), we 
introduce the polynomial
$$Q(\underline\Sigma, n) = \sum_{x\in \underline\Sigma}\# \Hom_{\Bcat} (\Oc_{\Sigma, x}^\times, \funn)\,,$$
and prove (Proposition \ref{finalproposition}) that $Q(\underline\Sigma, n) \leq P(\underline\Sigma, n)$.

Finally, we would like to emphasize that our approach to blueprints, being entirely functorial, seems to be  appropriate to 
carry out a ``derived version'' of the category of $\Bcat$-schemes. In fact, in quite general terms, a definition of ``derived $\Bcat$-scheme'' could be obtained by replacing, in our definition of $\Bcat$-scheme,  the category $\Set$ (resp.~$\Set_\ast$) by the  category $\mathsf S$ of spaces (resp.~$\mathsf S_\ast$ of pointed spaces) and the notion of monoid object by that  of $\mathbb E_\infty$-algebra. This issue will be the object of future work.

\smallskip\noindent
{\bf Acknowledgments.}   We would like to thank  an anonymous referee for pointing out a couple of mistakes in 
a previous version of this paper and for making helpful remarks.


\section{The general setting}\label{SectionTV}
\subsection{Schemes over a monoidal category}\label{relativeschemes}
For the reader's convenience, we start by giving a quick r\'esum\'e of some of the basic constructions of the 
``relative algebraic geometry'' developed in \cite[\S 2]{TV}.

Let $\Ccat =(\Ccat,\otimes, \mathbf{1})$ be a symmetric monoidal category ($\mathbf{1}$ is the unit object), and denote by 
$\mon_\Ccat$ the category of commutative (associative and unitary) monoid objects in $\Ccat$.

We assume that $\Ccat$ is {\em complete, cocomplete, and closed} (i.e., for every pair of objects $X$, $Y$,  
the contravariant functor $\Hom_{\Ccat}(\trait \otimes X, Y)$ is represented by an ``internal hom'' set $\ihom(X,Y)$).

The assumptions on $\Ccat$ imply, in particular, that the forgetful functor
\begin{equation*} \vert \trait \vert\colon \mon_\Ccat \to \Ccat
\end{equation*}
admits a left adjont
\begin{equation}\label{generalleftadjoint} L\colon \Ccat \to \mon_\Ccat\,,
\end{equation}
which maps an object $X$ to the free commutative monoid object $L(X)$ generated by $X$.

For each commutative monoid $V$ in $\mon_\Ccat$ one may introduce the notion of $V$-module  (cf.~\cite[p.~478]{Jo}).
The category  $V\trait\Mod$ of such objects has a natural symmetric monoidal structure given by the ``tensor product'' $\otimes_V$; this structure turns out to be closed. Given a morphism $V \to W$ in  $\mon_\Ccat$, there is a change of basis functor 
$$\trait \otimes_V W \colon V\trait\Mod \to W\trait\Mod\,,$$
whose adjoint is the forgetful functor $W\trait\Mod \to V\trait\Mod$. Note that the category of commutative monoids in $V\trait\Mod$ --- i.e.~the category of {\em commutative $V$-algebras} --- is naturally equivalent to the category $V/{\mon_\Ccat}$.

The category  $\aff_\Ccat$ of {\em affine schemes over $\Ccat$} is, by definition, the category  $\mon^{\text{op}}_\Ccat$. Given an object $V$ in  $\mon_\Ccat$ the corresponding object in $\aff_\Ccat$ will be denoted by $\spec V$. 

To define, in full generality, the category of schemes over $\Ccat$ one follows the standard procedure of glueing together  affine schemes. To this end, one first endows $\aff_\Ccat$ with a suitable Grothendieck topology. Let us recall the general definition.
\begin{definition}\label{topology}
Let $\Gcat$ be any category. A {\it Grothendieck topology} on $\Gcat$ is the assignment to each object $U$ of $\Gcat$ of a collection of sets of arrows $\{U_{i} \rightarrow U \}$ called {\it coverings of} $U$ so that the following conditions are satisfied:
\begin{enumerate}
\item[i)] if $V \rightarrow U$ is an isomorphism, then the set $\{ V\rightarrow U\}$ is a covering;
\item[ii)] if $\{U_{i} \rightarrow U\}$ is a covering and $V\rightarrow U$ is any arrow, then there exist the fibered products $\{U_{i}\times_{U}V\}$ and the collection of projections $\{U_{i}\times_{U}V\rightarrow V \}$ is a covering;
\item[iii)] if $\{ U_{i} \rightarrow U\}$ is a covering and for each index $i$ there is a covering $\{ V_{ij} \rightarrow U_{i}\}$ (where $j$ varies in a set depending on $i$), each collection $\{ V_{ij} \rightarrow U_{i}\rightarrow U\}$ is a covering of $U$. 
\end{enumerate}
A category with a Grothendieck topology is a called a {\it site}.
\end{definition}
\begin{rmark}
As it is clear from the definition above, a Grothendieck topology on a category $\Gcat$ is introduced with the aim of glueing  objects locally defined, and what really matters is therefore the notion of covering. So, in spite of its name, a Grothendieck topology could better be thought of  as a generalization of the notion of covering rather than of the notion of topology (notice, for example, that, though the maps $U_i\to U$ in a covering can be seen as a generalization of open inclusions $U_i\subset U$, no condition generalizing the topological requirement about unions of open subsets is prescribed).
\end{rmark}
Given a site $\Gcat$ and a covering $\mathcal{U}=\{ U_i\to U\}_{i\in I}$, we denote by $h_U$ the presheaf represented by $U$ and by $h_\mathcal{U}\subset h_U$ the subpresheaf of those maps that factorise through some element of $\mathcal{U}$.
\begin{definition}\label{sheaf}
Let $\Gcat$ be a site. A presheaf $F\colon\Gcat^{\text{\rm op}} \to\Set$ is said to be a sheaf if, for every covering $\mathcal{U}=\{ U_i\to U\}_{i\in I}$, the restriction map $\Hom (h_U,F)\to\Hom (h_\mathcal{U},F)$ is an isomorphism.
\end{definition}
Coming back to our symmetric monoidal category $\Ccat$, the associated category of affine schemes $\aff_\Ccat$ can be equipped with two different Grothendieck topologies by means of the following ingenious definitions (which, of course, generalize the corresponding  usual definitions   
 in ``classical'' algebraic geometry).

One says \cite[Def.~2.9, 1), 2), 3)]{TV} that a morphism $f\colon \spec W \to \spec V$ in  $\aff_\Ccat$ is
\begin{itemize}
\item {\em flat} if the functor $\trait \otimes_V W \colon V\trait\Mod \to W\trait\Mod$ is exact;
\item {\em an epimorphism} if, for any $Z$ in $\mon_\Ccat$, the functor 
$$f^\ast\colon \Hom_ {\mon_\Ccat}(W, Z) \to \Hom_ {\mon_\Ccat}(V, Z)$$ 
is injective ;
\item {\em of finite presentation} if, for any filtrant diagram $\{Z_i\}_{i\in I}$ in $V/{\mon_\Ccat}$, the natural morphism
$$\dirlim \Hom_{V/{\mon_\Ccat}}(W, Z_i) \to \Hom_{V/{\mon_\Ccat}}(W, \dirlim Z_i) $$
is an isomorphism.
\end{itemize}
\begin{definition}\label{definitionsofcovers} \cite[Def.~2.9, 4); Def.~2.10]{TV} a) A collection of morphisms 
$$\{ f_j\colon \spec W_j \to \spec V\}_{j\in J}$$ 
in 
$\aff_\Ccat$ is a flat cover if 
\begin{enumerate}
\item[i)] each morphism $f_j\colon \spec W_j \to \spec V$ is flat and
\item[ii)] there exists a finite subset of indices $J'\subset J$ such that the functor
$$\prod_{j\in J'} \trait \otimes_V W_j \colon V\trait\Mod \to \prod_{j\in J'}W_j \trait\Mod$$
is conservative.
\end{enumerate}
(b) A morphism $f\colon \spec W \to \spec V$ in  $\aff_\Ccat$ is an open Zariski immersion if it is a flat epimorphism of finite presentation.

\noindent
(c) A collection of morphisms $\{ f_j\colon \spec W_j \to \spec V\}_{j\in J}$ in 
$\aff_\Ccat$ is a Zariski cover if it is a flat cover and each $f_j\colon \spec W_j \to \spec V$ is an open Zariski immersion.
\end{definition}

\begin{rmark}\label{mainremarkonTV} The previous definition is actually a particular case of a more general  construction. Indeed, as shown in \cite{TV}, to define a topology on a complete and cocomplete category $\Dcat$ is enough to assign a pseudo-functor
$M\colon \Dcat^{\text{op}} \to \mathsf{Cat}$ satisfying the the following conditions:
\begin{enumerate}
\item[i)] for each morphism $q\colon X \to Y$ in $\Dcat$, the functor $M(q) = q^\ast\colon M(Y) \to M(X)$ has a right adjoint $q_\ast \colon M(X) \to M(Y)$ which is conservative
\item[ii)] for each Cartesian diagram
$$\xymatrix{ X' \ar[d]_r \ar[r]^{q'} & Y'\ar[d]^{r'} \\
X \ar[r]_q & Y}
$$
in $\Dcat$, the natural transformation $ q^\ast r'_\ast \Longrightarrow r_\ast q'^\ast$ is an isomorphism.
\end{enumerate}
In terms of such a functor one can define the notion of $M$-faithfully flat cover \cite[Def.~2.3]{TV} and the associated pretopology \cite[Prop.~2.4]{TV}, which induces a topology on $\Dcat$.

In the classical theory of schemes,  $\Dcat$ is the category $\Ring^{\text{op}}$ of affine schemes and, for each $X=\spec A$, $M(A)$ is the category of quasi-coherent sheaves on $X$. When starting with a monoidal category $\Ccat$ satisfying our assumptions, $\Dcat$ is the category $\aff_\Ccat$ and the pseudo-functor $M$ maps an object $V$ in $\mon_\Ccat$
to the category of $V$-modules and a morphism $\spec V \to \spec W$ to the functor $\trait \otimes_V W\colon V\trait\Mod\to W\trait\Mod$. What we have  called ``flat cover'' correspond to To\"en-Vaqui\'e's ``$M$-faithfully flat cover''  (cf.~\cite[Def.~2.8, Def.~2.10]{TV}).\\
When $\Dcat$ is endowed with a topology, a natural question that arises is how the pseudo-functor $M$ behaves with respect to it. It can be proven (\cite[Th. 2.5]{TV} that $M$ is a stack with respect to that topology (for the reader's convenience,  the notion of stack will be briefly reviewed in Appendix A).
\end{rmark}

By making use of flat covers and Zariski covers introduced in Definition \ref{definitionsofcovers} we may equip the category $\aff_\Ccat$ with two distinct Grothendieck topologies,
called, respectively, the {\em flat} and the {\em Zariski} topology. Correspondingly, there are two categories of sheaves on 
$\aff_\Ccat$, namely
$$\sh^{\text{flat}}(\aff_\Ccat) \subset  \sh^{\text{Zar}}(\aff_\Ccat)\subset \psh (\aff_\Ccat)\,.$$
Notice that, for each affine scheme $\Xi$,  the presheaf $Y(\Xi)$ given by the Yoneda embedding $Y(\trait)\colon \aff_\Ccat \to  \psh (\aff_\Ccat)$ is actually a sheaf in $\sh^{\text{flat}}(\aff_\Ccat) \subset  \sh^{\text{Zar}}(\aff_\Ccat)$ \cite[Cor.~2.11, 1)]{TV}; this sheaf will be denoted  again by $\Xi$.

The next and final step is to define the category of schemes over the category $\Ccat$. We first have to introduce the notion of  affine Zariski  cover in the category $\sh^{\text{Zar}}(\aff_\Ccat)$.
\begin{definition}\label{definitionsofcoversforsheaves} \cite[Def.~2.12]{TV} a) Let $\Xi$ be an affine scheme in $\aff_\Ccat$. A subsheaf $\cF\subset \Xi$
is said to be a Zariski open of $\Xi$ if there exists a collection of open Zariski immersions $\{\Xi_i \to \Xi\}_{i\in I}$ such that $\cF$ is the image of the sheaf morphism $\coprod_{i\in I} \Xi_i \to \Xi$.

\noindent
(b) A morphism $\cF \to \cG$ in  $\sh^{\text{Zar}}(\aff_\Ccat)$ is said to be an open Zariski immersion if, for any affine scheme $\Xi$ and any sheaf morphism $\Xi \to \cG$, the induced morphism 
$\cF \times_{\cG} \Xi \to \Xi$ is a monomorphism whose image is a Zariski open of  $\Xi$.

(c) Let $\cF$ be a sheaf  in $\sh^{\text{Zar}}(\aff_\Ccat)$. A collection of open Zariski immersions 
$\{\Xi_i \to \cF\}_{i \in I}$, where each $\Xi_i$ is an affine scheme over $\aff_\Ccat$, is said to be an affine Zariski cover of $\cF$ if the resulting
morphism 
$$\coprod_{i\in I} \Xi_i \to \cF$$
is a sheaf epimorphism.
\end{definition}  
It should be noted that, in the case of affine schemes over $\Ccat$, the definition of open Zariski immersion in Definition ~\ref{definitionsofcoversforsheaves}, (b) does coincide with that previously introduced in Definition ~\ref{definitionsofcovers}, (b) \cite[Lemma 2.14]{TV}.

\begin{definition}\label{defschemeoverC}
A scheme over the category $\Ccat$ is a sheaf $\cF$ in $\sh^{\text{Zar}}(\aff_\Ccat)$ that admits an affine 
Zariski cover.
The category of schemes over $\Ccat$ will be denoted by $\sch_\Ccat$.
\end{definition}

\subsection{Notation and examples}\label{Sectionnotation}
            
 Primarily to the purpose of fixing our notational conventions, we now briefly describe
 the basic examples of symmetric monoidal categories we shall work with in the sequel of the present paper.

\noindent
\textbullet\ The  category  $\Set$ of sets can be endowed with a monoidal product given by the Cartesian product. Then $(\mathsf{Set}, \times, \ast)$  is  a symmetric monoidal  category and  $\mon_\Set = \Mon$ is the usual category of commutative, associative and unitary monoids.

\noindent
\textbullet\  The category $\Set_\ast$ of pointed sets can be endowed with a monoidal product given by the smash product $\wedge$; in this case, the unit object is the pointed set $\mathbb{S}^0$ consisting of two elements. Then $(\Set_\ast, \wedge, \mathbb{S}^0)$ is a symmetric monoidal  category and  $\mon_{\Set_\ast} = \Mon_0$ is  the category of commutative, associative and unitary monoids with ``absorbent object'' (such an object will be denoted by $0$ in multiplicative notation and by $-\infty$ in additive notation).

\noindent
\textbullet\ The category $\Mon$ can be endowed with a monoidal product $\otimes$ defined in the following way:  $R\otimes R'$ is the quotient of the product $R\times R'$ by the relation $\mathcal{\sim}$ such that $(nr,r')\sim (r,nr')$ for each $(n,r,r')\in\mathbb{N}\times R\times R'$. Clearly, the unit object is the additive monoid $(\N, +)$. 
Then $(\Mon, \otimes, \N)$ is a symmetric monoidal  category and $\mon_\Mon= \SRing$ is the  category of commutative, associative and unitary semirings. 

\noindent
\textbullet\ The category $\Ab = \Z\trait \Mod$ of Abelian groups  can be endowed with a monoidal product $\otimes_\Z$ given by the usual tensor product of $\Z$-modules. Then  $(\Ab, \otimes_\Z, \Z)$ is a symmetric monoidal  category and $\mon_\Ab= \Ring$ is the  category of commutative, associative and unitary rings.

For the functor $L\colon \Ccat \to\mon_\Ccat $ defined in eq.~\ref{generalleftadjoint} as left adjoint to the forgetful functor
$\vert \trait \vert\colon \mon_\Ccat \to \Ccat$ we shall adopt the following special conventions:

%

\noindent
\textbullet\  if $\Ccat= \Set$, $L$ 
will be denoted by
\begin{equation} \N [\trait] \colon \Set \to \Mon\,;
\end{equation}

\noindent
\textbullet\  if $\Ccat= \Mon$, $L$
will be denoted by
\begin{equation} \trait \otimes_{{\mathbb U}} \N   \colon \Mon \to \SRing\,,
\end{equation}
where ${\mathbb U}$ is the monoid consisting of just one element (the notation being motivated by the identity
${\mathbb U} \otimes_{{\mathbb U}} \N = \N$);

\noindent
\textbullet\  if $\Ccat= \Mon_0$, $L$
will be denoted by 
\begin{equation} \label{adjointMon_0}\trait \otimes_{\fun}  \N  \colon \Mon_0 \to \SRing\,,
\end{equation}
where $\fun$ is the object of $\Mon_0$ consisting of two element, namely $\fun=\{0,1\}$ in multiplicative notation (also in this case,  the notation is motivated by the identity
$\fun \otimes_{\fun} \N = \N$); 

\noindent
\textbullet\  if $\Ccat= \Ab$, $L$
will be denoted by
\begin{equation} \Z [\trait] \colon \Ab \to \Ring\,.
\end{equation}

\smallskip

All symmetric monoidal categories $\Set$, $\Set_\ast$, $\Mon$, $\Mon_0$, $\Ab$ described above are complete, cocomplete, and closed, so we can apply the machinery of  To\"en-Vaqui\'e's theory illustrated in Subsection \ref{relativeschemes} and define, for each of these categories, the corresponding category of schemes over it. In this way, when $\Ccat= \Ab$, one unsurprisingly recovers the usual notion of {\em classical scheme}. A more intriguing example is provided by the case of
$\Ccat= \Set$.
%
%
%
%
%
%

\begin{ex} {\bf Monoidal schemes} An object of the category $\sch_\Set$ is a ``scheme over $\fun$'' in the sense of \cite{Dei}. The equivalence between the two definitions was proved in \cite{Vezz}. We recall that, if $M$ is a commutative monoid, its ``spectrum over $\fun$'' $\spec M$ can be realized as the set of prime ideals of $M$ and given a topological space
structure. 

In the present paper we shall call an object in $\sch_\Set$  a {\em monoidal scheme} and use the name of  ``$\fun$-scheme'' for a different kind of algebro-geometric structures (see Definition ~\ref{CCschemes}).
\end{ex}


\section{The category of blueprints}\label{SectionBschemes}

The notion of {\em blueprint} was introduced by Olivier Lorscheid in his 2012 paper \cite{Lor12}. 

\begin{definition}\label{def-blueprints}  A {\em blueprint} is a pair $B=(R, A)$, where $R$ is a semiring  and $A$ is a multiplicative subset  of $R$ containing $0$ and $1$ and generating $R$ as a semiring.
A blueprint morphism $f \colon B_1=(R_1, A_1) \to B_2=(R_2, A_2)$
is a semiring morphism $f\colon R_1 \to R_2$ such that $f (A_1) \subset A_2$. 
\end{definition}
Notice that, given a blueprint morphism  $f \colon B_1=(R_1, A_1) \to B_2=(R_2, A_2)$, its restriction $f\vert_{A_1}\colon A_1 \to A_2$ is a monoid morphism that uniquely determines $f$ on the whole of $R_1$. 


The idea underlying the notion of blueprint can be illustrated as follows. Some equivalence relations that do not make sense in a monoid $A$ may be expressed in the semiring $A\adj $. Now, any equivalence relation $\mathcal{R}$ on a semiring $S$ induces a projection 
$S\to S/\mathcal{R}$ and can indeed be recovered by such a map. So, the assignment of a pair $(A,A\adj \to R)$ is to be interpreted as the datum of a monoid $A$ plus the relation on $A\adj$ given by the epimorphism $A\adj \to R$.

\begin{ex}\label{example1} 
Consider the monoid $A_T =\mathbb{N}\cup\{ -\infty\}$ (in additive notation, corresponding to $\{ T^i\}_{i\in\N\cup\{ -\infty\}}$ in multiplicative notation) and the corresponding free semiring $A_T\adj$ of polynomials in $T$ with coefficient in $\N$ (the functor $\trait \adj$ has been introduced in eq.~\ref{adjointMon_0}). Notice that $\spec A_T$ has two points, namely the prime ideals $\{ -\infty\}$ and $(\N\setminus \{0\})\cup  \{ -\infty\}$, which embed in $\spec A_T\adj $ (we are loosely thinking of $\spec A_T\adj$ as the underlying topological space).\\
Now, if one takes a closed subset of $\spec A_T\adj $ and intersects it with $\spec A_T$, one could naively think that the intersection is nonempty only when the chosen closed subset is defined by some relation in $A_T$. However, this is not the case: for instance, the relation $2T=1$, which makes the ideal $(T)$ trivial,  cannot be expressed in the monoid $A_T$. According to Lorscheid's idea,  one  can represent this affine ``monoidal scheme'' by considering the pair $(A_T, A_T\adj \to A_T\adj /(2T=1))$.
\end{ex}

The category of blueprints can be given a handier description, which makes it easier to characterise it as the category of commutative monoids in a suitable symmetric monoidal category.  

Let us consider the functor $\trait \adj \colon \Mon_0 \to \SRing$ (introduced in eq.~\ref{adjointMon_0}) 
\begin{definition}\label{def-bluep} 
The category $\bluep$ is the full subcategory of  
$\trait \adj /\SRing$  whose objects 
$(A, A\adj \to R)$ 
satisfy the conditions:
\begin{equation}\label{Blueprintsconditions}
\begin{array}{l}
\text{a) the morphism $A\adj \to R$ is an epimorphism;}\\
\text{b) the composition\ }  
A\to \vert A\adj  \vert \to \vert R\vert \text{, is a monomorphism} \\ \text{\phantom{b)}(the first map being the unit of the adjunction).} 
\end{array}
\end{equation}
\end{definition}
It is immediate that  the category $\bluep$ is equivalent to the category of blueprints introduced in Definition \ref{def-blueprints}
 

Consider now the forgetful functor $\vert \trait \vert \colon \Mon_0 \to \Set_\ast$; for each monoid $M$ with absorbent object $0$ (in multiplicative notation), the base point of the associated set $\vert M\vert$ is clearly the element corresponding to $0$. Its adjoint functor is the functor 
\begin{equation*}
\N[\trait] \colon \Set_\ast \to\Mon_0\,.
\end{equation*}
We can now form the full subcategory $\Bcat$ of $\N[\trait] /\Mon_0$ whose objects $(X, \N[X] \to M)$ are described by conditions formally identical to those in eq.~\ref{Blueprintsconditions}
\begin{equation}\label{Bconditions}
\begin{array}{l}
\text{a) the morphism $\N[X] \to M$ is an epimorphism;}\\
\text{b) the composition\ }  
X\to \vert \N[X] \vert \to \vert M\vert\  \text{is a monomorphism.}
\end{array}
\end{equation}

\begin{rmark} The category $\Bcat$ above corresponds to the category of pointed set endowed with a pre-addition structure, as described 
in \cite[\S 4]{Lor16}. 
\end{rmark}

\begin{thm}\label{B-category} The category $\Bcat$ carries a natural structure of symmetric monoidal category. Moreover, this structure is closed, complete, and cocomplete.
\end{thm}
\begin{proof}
In the category $\Bcat$ there is a natural symmetric monoidal product given by
\begin{equation}\label{monoidalproductB}
(X, \N[X]  \to M)\otimes (X', \N[X'] \to M')=(X\wedge X', \N[X\wedge X'] \to M\otimes M')\,,
\end{equation}
where the map
$ \N[X\wedge X'] \to M\otimes M'$
is the composition
$$
\N[X\wedge X']  \to \N[X]  \otimes \N[X']  \to M\otimes M'\,;
$$
the first morphism maps $n(x,x')$ to $n x\otimes x'$ and  is an isomorphism (in other words, the functor $\N[\trait]$   is  monoidal).

Since $M\otimes M'$ is generated as a monoid by elements of the form $x\otimes x'$, and since the two maps 
$\N[X] \to M$ and $\N[X'] \to M'$ are surjective, the map $\N[X\wedge X']  \to M\otimes M'$ is also surjective. Moreover, by the definition of tensor product in the category $\Mon$, for any $x,y\in X\setminus\{\ast\}$ and $x',y'\in X'\setminus\{\ast\}$ one has  $x\otimes x'=y\otimes y'$ if and only if $(x,x')=(y,y')$, so that the map
\[
X\wedge X'\to |M\otimes M'|
\]
is a monomorphism. Conditions \ref{Bconditions} are therefore satisfied.

We now show that the monoidal category $\Bcat$ is closed. Let us define the internal hom functor by setting
\begin{equation}\label{internalhomB}
\ihom( (X, \N[X] \to M), (Y, \N[Y] \to N)) = 
(Y^X \times_{|N|^X} | N^{M}|,  \N[Y^X \times_{|N|^X} | N^{M}|]  \to\widetilde{N^M})\,,
\end{equation}
where $\widetilde{N^M}$ is the image of the map
$$
\N [Y^X \times_{|N|^X} | N^{M}|]\to \N [|N^{M}|] \to N^M
$$
(the second map above is the counit of the adjunction). Let us check the adjunction property.
For each map
\begin{equation}\label{Bmap}
(X,\N[X]\to M)\otimes (Y,\N [Y]\to N)=(X\wedge Y, \N[X\wedge Y]\to M\otimes N)\to (Z,\N[Z]\to L)\,,
\end{equation}
the first component corresponds, by the exponential law in $\Set_\ast$, to a map $X\to Z^Y$, while 
the second component is given by a commutative square
\begin{equation}\label{diagraminternalHom}
\xymatrix{
\N[X\wedge Y] \ar[r] \ar[d] & \N[Z] \ar[d] \\
M\otimes N \ar[r] & L
}
\end{equation}
where the  arrow on the left  is the product map $\N[X]\otimes \N[Y]\to M\otimes N$ and the top arrow  is the image of  the map in the first component through the functor $\N[\trait]$. 
By using  the property that  $\N[\trait]$ is the left adjoint to the forgetful functor and by noticing that
the bottom arrow in \ref{diagraminternalHom} corresponds to a map $M\to L^N$, it is immediate that assigning  
the commutative diagram \ref{diagraminternalHom} 
is equivalent to assigning the two commutative diagrams
\[
\xymatrix{
X \ar[r] \ar[dr] & Z^Y \ar[d] \\
 & |L|^Y
} \qquad \quad
\xymatrix{
X \ar[d] \ar[dr] & \\
|M| \ar[r] & |L^N|
}
\]
together with the condition that the diagonal morphism of the first coincides with the composition of the diagonal morphism of the second  
and the morphisms $|L^N|\hookrightarrow |L|^{|N|}\to |L|^Y$ (the second map being induced by  
the map $Y\to |N|$).  
Summing up, a map as in eq.~\ref{Bmap} is equivalent to a map from $X$ to the pullback defined by the diagram
\begin{equation*}
\xymatrix{
& Z^Y \ar[d] \\
|L^N| \ar[r] & |L|^Y
}
\end{equation*}
along with a  compatible map $M\to L^N$ in such a way that the following diagram commutes:

\begin{equation*}
\xymatrix{
X \ar[r] \ar[d] \ar@/^1.5pc/[rr] & \vert L^N\vert \times_{\vert L\vert ^Y} Z^Y\ar[r] \ar[d] & Z^Y \ar[d]\\
|M| \ar[r] & \vert L^N\vert \ar[r] & \vert L\vert ^Y
}
\end{equation*}
This shows that the internal hom functor in eq.~\ref{internalhomB} is indeed a right adjoint to the monoidal product functor in eq.~\ref{monoidalproductB}. 

We wish now to show that the category $\Bcat$ is complete and cocomplete.  First we prove that it admits colimits. Given 
a diagram whose objects are $(X_i,\N [X_i]\to M_i)$,  we claim that its colimit is the object
\begin{equation*}
B=(\widetilde{\dirlim X_i},\N [\widetilde{\dirlim X_i}]\to\dirlim M_i)\,,
\end{equation*}
where $\widetilde{\dirlim X_i}$ denotes the image of the natural map $\dirlim X_i\to |\dirlim M_i|$; the maps from the diagram to $B$ are the obvious ones. 
It is immediate that $B$ is  an object of $\Bcat$. The injectivity condition is satisfied by definition. As for the surjectivity condition, one has that, since the functor $\N [\trait ]$ preserves colimits (being a left adjoint), the map $\N [\dirlim X_i]\to\dirlim M_i$ is surjective (by \cite{CWM}, Theorem V.2.1, it is enough to show that for arbitrary coproducts and coequalizers, in which cases it is a consequence of the surjectivity of the maps $\N [X_i]\to M_i$), so that the image of $\dirlim X_i$ generates $\dirlim M_i$; hence, the map $\widetilde{\N [\dirlim X_i]}\to\dirlim M_i$ is surjective.

Consider a map from the given diagram to an object $C$ of $\Bcat$. In the  category $\N [\trait]/\Mon_0$ such a map factorises in a unique way through the object $(\dirlim X_i,\N [\dirlim X_i]\to\dirlim M_i)$ because of the colimit properties  in the categories $\Set_\ast$ and $\Mon_0$ and because the functor $\N [\trait ]$ preserves colimits. If two elements 
  $x,y\in\dirlim X_i$ have the same image $m\in\dirlim M_i$, then their images in the first component of $C$ are mapped by the morphism in the second component to the same element. So, the images of $x$ and $y$ do coincide, just because $C$ is an object of $\Bcat$. It follows that the map from the diagram in $C$ uniquely factorises through $B$, so that our claim is proved. 

Second we prove that $\Bcat$ admits limits. Given a diagram as above, we claim  that its limit is the object
\begin{equation*}
B'=(\invlim X_i,\N [\invlim X_i]\to\widetilde{\invlim M_i})\,,
\end{equation*}
where $\widetilde{\invlim M_i}$ is the image of the natural map $\N [\invlim X_i]\to\invlim M_i$, which is adjoint to the map $\invlim X_i\to\invlim |M_i|\cong |\invlim M_i|$ (the last isomorphism holds since $|\trait|$ preserves limits, being a right adjoint) induced by the maps $X_i\to |M_i|$; the maps from $B'$ to the diagram are the obvious ones. 
It is clear that $B'$ is an object of $\Bcat$: the surjectivity condition holds by definition, while for the injectivity condition it is enough to note that it holds when the limit is either an arbitrary product or an equalizer (see \cite{CWM}, Theorem V.2.1). 
Consider now a map from an object $C$ to the given diagram. In the  category $\N [\trait ]/\Mon_0$ such a map uniquely factorises through the object $(\invlim X_i,\N [\invlim X_i ]\to\invlim M_i)$, because of the limit properties in the categories $\Set_\ast$ and $\Mon_0$. Since the second component of $C$ is a surjective morphism, this map uniquely factorises through $B'$. Thus, $B'$ satisfies the limit condition, as claimed. 
\end{proof}

\begin{prop}\label{Bluep-category} The category $\bluep$ of blueprints is equivalent to the category $\mon_\Bcat$ of monoids in the symmetric monoidal category
$\Bcat$.
\end{prop} 
\begin{proof}
To begin with, notice that, for each monoid object $((X,\N[X] \to M),\mu )$ in $\Bcat$, 
the domain of the multiplication map
\[
\mu\colon (X,\N[X] \to M)\otimes (X,\mathbb{N}[X]\to M)\to (X,\N[X]\to M)
\]
is defined in eq.~\ref{monoidalproductB} as
\[
(X,\N[X] \to M)\otimes (X,\mathbb{N}[X]\to M):=(X \wedge X, \N[X\wedge X]\to M\otimes M)\,.
\]
So the first component of $\mu$
is a map 
\[
m:X\wedge X\to X
\]
which defines a (multiplicative) monoid structure on the set $X$, while the second component of $\mu$ 
yields a commutative diagram 
\[
\xymatrix{
\N[X\wedge X] \ar[d] \ar[r]^{\N[m]} & \N[X] \ar[d] \\
M\otimes M \ar[r] & M
}
\]
whose bottom arrow induces an associative and commutative multiplication on the monoid $M$ compatible with its monoidal sum; in other words, it induces  a semiring structure on $M$.

Similarly, the top arrow induces a semiring structure one the monoid $\N[X]$. In this case, since the multiplication is given by the application of the free monoid functor $\mathbb{N}[\trait]$ to the multiplication $m$ of $X$, the resulting semiring is nothing but the free semiring $X \adj $ 
generated by the monoid $(X, m)$. The commutativity of the diagram ensures that the multiplication on $X$ is consistent with that on $M$, so that $X$ can still be seen as a subobject of $|M|$.

In conclusion, a monoid object in the category $\Bcat$ is a blueprint, and it is also obvious that any blueprint can be obtained this way.
\end{proof}

\begin{rmark} Theorem \ref{B-category} and Proposition \ref{Bluep-category} should hopefully provide a full elucidation of   \cite[Lemma 4.1]{Lor16}.
\end{rmark}

We have shown that the category of blueprints fits in with the general framework proposed by To\"en and Vaqui\'e,
so we can apply the formalism of Subsection \ref{relativeschemes} to define the category of schemes over $\Bcat$. 
\begin{definition}
An affine $\Bcat$-scheme is an object of the category $\aff_\Bcat = \bluep^{\text{op}}$, a $\Bcat$-scheme an object of the category $\sch_\Bcat$ {\em (see Definition \ref{defschemeoverC})}.
\end{definition}
\begin{rmark}A ``$\Bcat$-scheme'' corresponds to  what is called a ``subcanonical blue scheme'' in \cite{Lor16}.
\end{rmark}
%
%
%
%


\section{Adjunctions}\label{sectionadjunctions}

\subsection{$\Bcat$-schemes}\label{sectionBadjunctions}
This sections aims to show that the natural adjunction between the categories $\aff_{\Mon_0}$ and $\aff_{\Set_\ast}$ factorizes
through an adjunction between the categories $\aff_{\Mon_0}$ and $\aff_{\Bcat}$ and an adjunction between 
the categories $\aff_{\Set_\ast}$ and $\aff_{\Bcat}$, whose right adjoints induce functors between the corresponding categories of relative schemes. 
%
%
%

%
%
\begin{lemma}\label{blueforgetfulfunctorlemma}
The  functor   $\tilde{F}\colon \N[\trait]  / \Mon_0 \to \Mon_0$
mapping an object $(X, \N[X]  \to M)$ to the monoid $M$ admits a right adjoint 
\begin{equation}\label{preblueforgetfulfunctor}
\tilde{G} \colon \Mon_0 \to \N[\trait]  / \Mon_0\,,
\end{equation}
mapping 
a monoid $M$ to the object $(|M|, \N[|M|] \to M)$, where the second component is the counit of the adjunction $\N[\trait]  \dashv \vert\trait \vert$. 
The adjunction $\tilde{F} \dashv \tilde{G}$ induces an adjunction between the associated categories of monoids
\begin{equation}\label{preblueadjunction}
\xymatrix{\SRing \ar@/_1.1pc/[r]^{G} & {\trait\adj  / \SRing } \ar@/_1.1pc/[l]_{F}}\,,
\end{equation}
where $F$
maps an object $(A, A\adj \to R)$ to the semiring $R$ and its  right adjoint $G$
maps
a semiring $R$ to the object $(|R|, |R|\adj \to R)$, where the second component is the counit of the adjunction $\trait\adj  \dashv \vert\trait \vert$. 
\end{lemma}
\begin{proof}
Let   $(X, \N[X]  \to M)$ be an object of $\N[\trait]  / \Mon_0$
and $N$ a monoid. Let us consider a morphism  
$$
(X,\mathbb{N}[X]\to M)\to (|N|,\mathbb{N}[|N|]\to N)
$$
in the category $\N[\trait]  / \Mon_0$ and 
denote by $f\colon X\to |N|$ the induced set morphism.
In the commutative square
\begin{equation}\label{squareofpreblueadjunction}
\xymatrix{
\mathbb{N}[X] \ar[r]^{\N[f]} \ar[d] & \mathbb{N}[|N|] \ar[d] \\
M \ar[r] & N
}
\end{equation}
the  map $\N[f]$, because of the property of the vertical arrow on the right (which is the counit of the adjunction), amounts to the same as a map
$
\mathbb{N}[X]\to N
$.
Such a map, by adjunction, must be induced by the map $f\colon X\to |N|$. Thus, the assignment of the map $f$ and the commutative square \ref{squareofpreblueadjunction} are equivalent to the assignment of the commutative triangle
\begin{equation*}
\xymatrix{
\mathbb{N}[X] \ar[dr] \ar[d] & \\
M \ar[r] & N
}
\end{equation*}
But this diagram is equivalent to the assignment of a map $M\to N$, since the vertical map is given. 
We have therefore the adjunction $\tilde{F} \dashv \tilde{G}$, as claimed. The last statement is now straightforward. \end{proof}

Since image of the functor $\tilde{G}\colon \Mon_0 \to \N[\trait]  / \Mon_0$ is contained in the subcategory $\Bcat$,  
the adjunction \ref{preblueadjunction} restricts to the adjunction 
\begin{equation}\label{blueadjuction} 
\xymatrix{\SRing \ar@/_1.1pc/[r]^{G} & {\bluep} \ar@/_1.1pc/[l]_{F}}\,.
\end{equation}

It is immediate that the adjunction $\xymatrix{\SRing \ar@/_1.1pc/[r]^{\vert\trait\vert} &\Mon_0 \ar@/_1.1pc/[l]_{\trait\adj}}$ factorises through the adjunction \ref{blueadjuction} and the adjunction
\begin{equation}\label{MonBluepadjuction}
\xymatrix{\Mon_0 \ar@/_1.1pc/[r]^{\sigma} & {\bluep} \ar@/_1.1pc/[l]_{\rho}}\,,
\end{equation}
where $\rho(A,A\adj \to R) = A$ and 
$\sigma(A) = (A,\xymatrix{A\adj \ar[r]^= &A\adj})$.


The adjunctions above induce opposite adjunctions between the corresponding categories of affine schemes. We have therefore the following diagram 
\begin{equation}\label{adjunctiondiagram}
\xymatrix{
\aff_{\Mon_0}  \ar[r]^{\vert\trait\vert} \ar[d]^{G} & \aff_{\Set_\ast}   \ar@/_1.5pc/[l]_{\trait\adj} \ar[dl]^{\sigma} \\
\aff_{\Bcat} \ar@<1ex>[ur]^{\rho} \ar@/^0.8pc/[u]^{F}
}
\end{equation}
associated to the diagram
\begin{equation}\label{adjunctiondiagram2}
\xymatrix{
\Mon_0  \ar[r]^{\vert\trait\vert} \ar[d]^{\tilde{G}} & \Set_\ast   \ar@/_1.5pc/[l]_{\N[\trait]} \ar[dl]^{\tilde{\sigma}} \\
\Bcat \ar@<1ex>[ur]^{\tilde{\rho}} \ar@/^0.8pc/[u]^{\tilde{F}}
}
\end{equation}

We now wish to show that the functors in  diagram \ref{adjunctiondiagram2} satisfy the conditions that are required to apply  
\cite[Prop.~2.1, Cor.~2.2]{TV}. Of course, it will be enough to check that for the adjunctions  $\tilde{ F} \dashv \tilde{G}$ and $\tilde{\rho} \dashv \tilde{\sigma}$.
 
\begin{lemma}\label{lemmaadjunction1} In the adjunction $\xymatrix{\Mon_0 \ar@/_1.1pc/[r]^{\tilde{G}} & {\Bcat} \ar@/_1.1pc/[l]_{\tilde{F}}}$ 
\begin{enumerate}
\item the left adjoint $\tilde{F}$ is monoidal;
\item the right adjoint $\tilde{G}$ is conservative;
\item the functor $\tilde{G}$ preserves filtered colimits.
\end{enumerate} 
\end{lemma}
\begin{proof}
(1) and (2) are straightforward.

As for (3), we have to show that the right adjoint preserves filtered colimits, which is also quite obvious. The colimit of a filtered diagram $(X_i, \N[X_i] \to M_i)$ is indeed given by
\[
(\underrightarrow{\mathrm{lim}}X_i, \N[\underrightarrow{\mathrm{lim}}X_i] \to\underrightarrow{\mathrm{lim}}M_i)
\]
provided that it belongs to our category (notice that $\N[\underrightarrow{\mathrm{lim}}X_i] \cong \underrightarrow{\mathrm{lim}}\N[X_i]$ since $\N[\trait]$ is a left adjoint). But it does, because the map 
$N[\underrightarrow{\mathrm{lim}}X_i] \to\underrightarrow{\mathrm{lim}}M_i$ is surjective due to the fact that so are
the maps $\N[X_i]\to M_i$ and the injectivity condition is satisfied since the diagram is filtrant. 
\end{proof}

\begin{lemma}\label{lemmaadjunction2} In the adjunction $\xymatrix{\Bcat \ar@/_1.1pc/[r]^{\tilde{\rho}} & {\Set_\ast} \ar@/_1.1pc/[l]_{\tilde{\sigma}}}$ 
\begin{enumerate}
\item the left adjoint $\tilde{\sigma}$ is monoidal;
\item the right adjoint $\tilde{\rho}$ is conservative;
\item the functor $\tilde{\rho}$ preserves filtered colimits.
\end{enumerate} 
\end{lemma}
\begin{proof}
The functors $\tilde{\sigma}$, $\tilde{\rho}$ are defined as follows:  $\tilde{\sigma}(X) = (X, \xymatrix{\N[X] \ar[r]^{=} &\N[X]})$ and $\tilde{\rho}(X, \N[X] \to M) = X$. (1) is then straightforward. As for (2), we know that a map $(X,\N[X]\to M)\to (Y,\N[Y]\to N)$ is determined by the first component, so that $\tilde{\rho}$ is conservative. Finally, (3) is proved by proceeding as in the proof of Lemma \ref{lemmaadjunction1}.
\end{proof}

\begin{prop}\label{Bschemesadjunctions}The functor $F\colon \aff_{\Bcat} \to \aff_{\Mon_0}$ is continuous w.r.t.~the Zariski and the flat topology; morevover,
the functor 
\begin{equation} \widehat{F} \colon \sh(\aff_{\Bcat}) \to \sh(\aff_{\Mon_0})
\end{equation}
preserves the subcategories of schemes and so induces a functor
\begin{equation}\label{Bschemesadjunctionseq}
\begin{aligned} \widehat{F}\colon \sch_{\Bcat} &\to \sch_{\Mon_0}\\
\Sigma &\mapsto \widehat{F}(\Sigma)
\end{aligned}
\end{equation}
\end{prop}
\begin{proof} 
a) We first note that, given objects  $X_M= (X, \N[X]\to M)$, $X_M'= (X, \N[X]\to M')$ in $\Bcat$,
if $X_M\to X_{M'}$ is a flat morphism in $\Bcat$, then in the associated diagram
\[
\xymatrix{
M\trait\Mod \ar[d] \ar[r] & X_M\trait\Mod \ar[d] \\
M' \trait\Mod \ar[r] & X_{M'}\trait\Mod
}
\]
the natural transformation between the two compositions is an isomorphism. We wish to prove that an analogous property
holds when one considers a flat morphism in the category $\bluep$. As usual, it will be enough to work in the category $\trait \adj /\SRing$. Let $A_R = (A, A\adj \to R)$ and $A_S =(A, A\adj \to S)$ be objects in this category, and consider 
a flat morphism $A_R\to A_S$. An $A_R$-module is given by a pair
$$
(N,M)\in\Set_\ast\times\Mon_0
$$
such that $N$ is a subset of $|M|$ and generates it as a module, together with an action of $A$ on $N$ and an action of $R$ on $M$, such that the former is the restriction of the latter.  If $M$ is an $R$-module $M$, its associated $A_R$-module is the $(R, R \adj  \to R)$-module $(\vert M\vert,M)$, whose $A_R$-module structure is induced  by the map
$$
A_R\to (R, R\adj \to R)
$$
given by the pair of immersions $\iota\colon A\hookrightarrow R$ and $\iota\otimes_{\fun}\text{id} \colon A \adj  \to R\adj$, where the latter fits
in the commutative square
$$
\xymatrix{
A \adj  \ar[d] \ar[r]^{\iota\otimes_{\fun}\text{id}} & R\adj  \ar[d] \\
R \ar[r]_{\text{id}_R} & R
}
$$
The category $R\trait\Mod$ can therefore be identified with the full subcategory of the category of
$$
(A\adj , A\adj \to R)\trait\Mod
$$
 whose underlying objects in $\Mon_0 /\Mon_0$ are of the kind $(M,M=M)$.
 
We have now to show that, for any flat morphism $A_R\to A_S$ in $\trait \adj /\SRing$,  in the associated diagram
$$
\xymatrix{
R\trait\Mod \ar[d] \ar[r] & A_R\trait\Mod \ar[d] \\
S\trait\Mod \ar[r] & A_S\trait\Mod
}
$$
the natural transformation between the two compositions is an isomorphism. As for the first component, the commutativity up 
isomorphism of the above diagram is straightforward. As for the second component, that can be easily shown by adapting the argument in proof of Prop.~3.6 of \cite{TV}. The statement then follows from \cite[Cor.~2.22]{TV}.
\end{proof}

\begin{prop}\label{Bschemesadjunctions2}
The functor $\sigma \colon \aff_{\Set_\ast} \to \aff_{\Bcat}$ is continuous w.r.t.~the Zariski and the flat topology; morevover,
the functor 
\begin{equation} \widehat{\sigma} \colon \sh(\aff_{\Set_\ast}) \to \sh(\aff_{\Bcat})
\end{equation}
preserves the subcategories of schemes and so induces a functor
\begin{equation}
\begin{aligned} \widehat{\sigma}\colon \sch_{\Set_\ast} &\to \sch_{\Bcat}\\
\Xi &\mapsto \widehat{\sigma}(\Xi)
\end{aligned}
\end{equation}
\end{prop}
\begin{proof}
Consider a flat morphism $A\to B$ in the category $\Mon_0$, and denote by $A_{A\adj}$ the object 
$(A, A\adj = A\adj)$ in $\trait\adj/\SRing$. Each $A_{A\adj}$-module is given by a pair 
$(N,M)\in\Set_\ast \times\Mon_0$ together with an action of $A$ on $N$ and an action of $A\adj $ on $M$, the two actions being compatible in the obvious sense. In the diagram
$$
\xymatrix{
A_{A\adj}\trait\Mod \ar[d] \ar[r] & A\trait\Mod \ar[d] \\
B_{B\adj}\trait\Mod \ar[r] & B\trait\Mod
}
$$
the horizontal map sends an object $(N,M)$ to the set $N$ endowed with an action of the monoid $A$. Since tensor products are defined ``componentwise'', the diagram commutes.\\
\end{proof}

%

\subsection{$\nbcat$-schemes}\label{SectionFinal1}

By Proposition \ref{Bschemesadjunctions2} there is an induced functor $\widehat{\sigma}\colon \sch_{\Set_\ast} \to \sch_{\Bcat}$. One would like this functor to have a left adjoint determined by the functor $\rho\colon \aff_{\Set_\ast} \to \aff_{\Bcat}$. The functor $\rho$ may be easily shown to preserve Zariski covers, but it does not commute
with finite limits (in other words, it is not continuous w.r.t.~the Zariski topology, according to the usual terminology).
\begin{ex}
Let  us consider the free monoid $M=\langle X,Y\rangle$ and the blueprint $B$ defined by the free monoid
$\langle T, T_1,T_2,S, S_1, S_2\rangle$ with the relations $T= T_1+T_2$ and $S= S_1 +S_2$. Let $f, g\colon M \to B$ be the morphisms mapping $(X, Y)$, respectively, into $(T_1, T_2)$ and $(S_1, S_2)$. The coequalizer of $f$ and $g$ is the blueprint $B'$ defined by the free monoid $\langle X,Y, Z\rangle$ with the relation $Z= X+Y$, while the coequalizer of $\rho f$ and $\rho g$ is the the free monoid $\langle T,S, Z_1, Z_2\rangle$. The latter is obviously different from $\rho B'$.
\end{ex}

This drawback may be sidestepped by proceeding as follows:
1) omit the requirement that the map $A\to \vert A\adj  \vert \to \vert R\vert$ is a monomorphism in Definition \ref{def-bluep} and define a  category $\nbluep$ that contains the category $\bluep$ of blueprints; analogously, by omitting the  second condition in eq.~\ref{Bconditions}, define a category $\nbcat$ containing $\Bcat$; 2) 
prove that   there is a  functor $\rho\colon \aff_{\Set_\ast} \to \aff_{\nbcat}$ that is continuous w.r.t.~the Zariski topology; 3) define the category of schemes $\sch_{\nbcat}$ associated to this new category; 4) restrict our attention to the subcategory of 
$\sch_{\nbcat}$ consisting of schemes that admit a cover by affine schemes in the category $\aff_{\Bcat}$.

More precisely, the categories $\nbcat$ and $\nbluep$ are defined in the following way.
\begin{definition}\label{def-nbluep} 
The category $\nbcat$ is the full subcategory  of $\N[\trait] /\Mon_0$ whose objects $$(X, \N[X] \to M)$$ satisfy the condition that the morphism $\N[X] \to M$ is an epimorphism.\\
The category $\nbluep$ is the category $\mon_\nbcat$ of monoids in the symmetric monoidal category $\nbcat$.
\end{definition}
 
We denote again by $\rho\colon \nbluep \to  \Mon_0$ the forgetful functor, $\rho(A,A\adj \to R) = A$; 
 analogously  to adjunction ~\ref{MonBluepadjuction}, there is an adjunction 
\begin{equation}\label{MonBluepadjuctionnew}
\xymatrix{\Mon_0 \ar@/_1.1pc/[r]^{\sigma} & {\nbluep} \ar@/_1.1pc/[l]_{\rho}}\,,
\end{equation}
where $\sigma(A) = (A,\xymatrix{A\adj \ar[r]^= &A\adj})$.

\begin{lemma}\label{extm}
{\rm (a)}
Given an object $(A,A\adj\to R)$ of $\nbluep$, any diagram $X\colon I\to A-\Mod$ can be lifted to a diagram $I\to (A,A\adj\to R)-\Mod$.\\
{\rm (b)} Given a diagram $X\colon I\to \Mon_0$ and  a sieve $I_0$ of $I$,   any lift of $X_{\vert I_0}$ to a diagram $I_0\to \nbluep$ can be extended to a diagram $I\to \nbluep$.
\end{lemma}
\begin{proof}
{\rm (a)} Let $X\colon I\to A-\Mod$ be a diagram.
For each object $i$ of $I$, consider the $(A,A\adj\to R)$-module $(X_i,\N [X_i]\to M_i^0)$,
 where $M_i^0$ is the quotient of $\N [X_i]$ by the equivalence relation generated by $am=bm$, for each $m\in\N [X_i]$ and for each 
 pair $(a,b)$ in the relation defining the quotient $R$.\\
By induction,  the $(A,A\adj\to R)$-module $(X_i,\N [X_i]\to M_i^{\alpha+1})$ is defined by setting
 $M_i^{\alpha+1}$ to be the quotient of $\N [X_i]$ by the equivalence relation generated by the equations defining $M_i^\alpha$ and by the equations $\N [f]m=\N [f]n$, where  $f\colon X_j\to X_i$ is any map in the diagram and where  $m=n$ w.r.t.~the relation defining $M_j^\alpha$. When $\alpha$ is a limit ordinal, $M_i^\alpha$ is defined as the obvious colimit 
 $\dirlim_{\beta <\alpha} M_i^\beta$. Finally, let $M_i = \dirlim_{\alpha} M_i^\alpha$. It is clear that the diagram $X$ can be lifted in a unique way to a diagram $(X_i,\N [X_i]\to M_i)$.\\
{\rm (b)} The proof is analogous to that of point (a).
\end{proof}
\begin{rmark}\label{extr}
A particular case of Lemma \ref{extm}(b) is the following. Given an object $(A,A\adj\to R)$ of $\nbluep$, any diagram $\xymatrix{A \ar@<.5ex>[r]^f \ar@<-.5ex>[r]_g &B}$ 
can be lifted (w.r.t.~$\rho$) to a diagram 
$\xymatrix{(A,A\adj\to R) \ar@<.5ex>[r] \ar@<-.5ex>[r]& (B,B\adj\to S)}$.
\end{rmark}

\begin{rmark} Should one admit the existence of the zero monoid  and of the zero ring (i.e.~the possibility that $0=1$), in the proof of Lemma \ref{extm} it would be enough to set $M_i = 0$ and $S=0$, respectively
\end{rmark}

\begin{prop}\label{rhopreservesZariski} The functor  $\rho\colon \aff_{\Set_\ast} \to \aff_{\nbcat}$ preserves Zariski covers.
\end{prop}
\begin{proof}
Let  
$$\left\{ \spec (A_i, A_i\adj\to R_i) \to \spec (A, A\adj \to R)\right\}_{i\in I}$$
be any  Zariski cover  in the category $\aff_\nbcat$.
We have to prove that $\{ \spec A_i\to\spec A\}_{i\in I}$ is a Zariski cover in $\aff_{\Set_\ast}$. To do that, by taking into account \cite[D\'ef.~2.10]{TV}, we have to check the following four points:
\begin{enumerate}
\item To show that, for each $i$, $\spec A_i\to\spec A$ is flat, that is  that
$$
\trait\otimes_AA_i\colon A-\Mod\to A_i-\Mod
$$
is exact. By applying Lemma \ref{extm}(a) to any finite diagram, this follows from the flatness of the morphism $\spec (A_i, A_i\adj\to R_i) \to \spec (A, A\adj \to R)$ and from the fact that $\rho$ preserves limits, being a right adjoint.
\item To show that there is a finite subset $J\subset I$ such that
$$
\prod_{j\in J}\trait\otimes_AA_j\colon A-\Mod\to\prod_{j\in J}A_j-\Mod
$$
is conservative. This follows  from  Lemma \ref{extm}(a) in the case where $I$ is the category $\bullet \to \bullet$. 
\item To show  that $\rho$ preserves epimorphisms. This is consequence of  Lemma \ref{extm}(b) (see Remark \ref{extr}). 
\item To show that $\rho$ preserves the finite presentation property. This fact follows from  Lemma \ref{extm}(b). \end{enumerate}
\end{proof}

\begin{prop}\label{rhopreservesfinitelimits}
The functor  $\rho\colon \aff_{\Set_\ast} \to \aff_{\nbcat}$ preserves finite limits.\end{prop}
\begin{proof}
We will show the equivalent statement that the opposite functor from $\rho\colon {\nbcat} \to {\Set_\ast}$ preserves finite colimits. As usual, it is enough to show that it preserves finite coproducts and coequalizers.\\
Let $(A,A\adj\to R)$ and $(B,B\adj\to S)$ be objects in  ${\nbcat}$ and take the coproduct $(A\coprod B,(A\coprod B)\adj\to R\oplus S)$ in the category $\trait \adj /\SRing$:  we have to show that the second component is surjective. This follows from the fact that, being $\trait\adj$ a left adjoint, one has $(A\coprod B)\adj\cong (A\adj)\oplus (B\adj)$.\\
Let $f,g\colon (A,A\adj\to R) \to (B,B\adj\to S)$. Analogously as above, the domain of the second component of the coequalizer $C$ of $f, g$ in $\trait \adj /\SRing$  is the coequalizer of
$$
f\adj ,g\adj\colon A\adj\to B\adj\,.
$$
Because of the universal property of colimits, there is a commutative diagram 
giving rise to a commutative diagram
$$
\xymatrix{
A\adj \ar@{>>}[d] \ar@<.5ex>[r] \ar@<-.5ex>[r] & B\adj \ar@{>>}[d] \ar@{>>}[r] & C \ar[d] \\
R \ar@<.5ex>[r] \ar@<-.5ex>[r] & S \ar@{>>}[r] & T}
$$
in $\SRing$, whose rows are coequalizers and where the map  
$C\to T$ is the second component of the coequalizer of $f,g$ in the category $\trait \adj /\SRing$. As the middle vertical map and the bottom right one are surjective, so is the map 
$C\to T$.
\end{proof}

Proposition \ref{rhopreservesZariski} and Proposition \ref{rhopreservesfinitelimits} entail the following result.
\begin{corol}
The  functor $\rho\colon \aff_{\Set_\ast} \to \aff_{\nbcat}$  is continuous w.r.t.~the Zariski topology, and 
the adjunction \ref{MonBluepadjuctionnew} gives rise to a geometric morphism
\begin{equation}\label{MonBluepadjuctionnew2} \xymatrix{
\sh(\aff_{\nbcat})  \ar@/_1.1pc/[r]^{\widehat\rho}  
& \sh(\aff_{\Set_\ast}) \ar@/_1.1pc/[l]_{\widehat\sigma}}
\end{equation}
\end{corol}

\begin{thm}\label{rhosigmaadjunction} The functor $\widehat\rho\colon \sh(\aff_{\nbcat}) \to \sh(\aff_{\Set_\ast})$
preserves the subcategories of schemes and so induces a functor 
\begin{equation}
\widehat\rho \colon \sch_{\nbcat} \to \sch_{\Set_\ast}\,.
\end{equation}
Hence, the adjunction \ref{MonBluepadjuctionnew2} induces an adjunction 
$\widehat\rho \dashv \widehat\sigma\colon \sch_{\nbcat} \to \sch_{\Set_\ast}$.
\end{thm}
\begin{proof}
We already proved that $\widehat\sigma$ preserves the relevant subcategory of schemes in Proposition \ref{Bschemesadjunctions2}. So all we have to prove is that $\widehat\rho$ preserves the relevant subcategory of schemes.
In view of \cite[Proposition 2.18]{TV}, it suffices to observe that the following properties of $\widehat\rho$ are satisfied:
\begin{itemize}
\item it preserves coproducts (for it is a left adjoint), and affine schemes;
\item it preserves finite limits (by Proposition \ref{rhopreservesfinitelimits}) and Zariski opens of affine schemes (by Lemma \ref{extm}(b) and by the fact that  $\widehat\rho$  preserves finite limits);
\item it preserves images (since it preserves finite limits and colimits) and diagonal morphisms;
\item it preserves quotients, since it preserves colimits.
\end{itemize}
\end{proof}

\begin{definition}\label{def-nbschemes}
 A scheme $\Sigma$ in $\sch_{\nbcat}$ that admits a Zariski  cover by affine schemes in $\aff_{\Bcat}$
will be called (by a slight abuse of language) a $\nbcat$-scheme. The category of such schemes will be denoted 
by $\widetilde{\sch_{\nbcat}}$.
\end{definition}
The rationale behind this definition is that, while $\nbcat$-schemes retain all good local properties of $\Bcat$-schemes (namely, the properties of blueprints), one gains the advantages of working  in the wider and more comfortable environment of the category 
$\sch_{\nbcat}$. 

Notice that the adjunction in Theorem \ref{rhosigmaadjunction} obviously restrict to an adjunction 
\begin{equation}\label{adjunctionfornbschemes}
\widehat\rho \dashv \widehat\sigma\colon \widetilde{\sch_{\nbcat}} \to \sch_{\Set_\ast}\,.
\end{equation}
Morevover, one can define a functor 
\begin{equation}\label{Zbluefunctor0}
\xymatrix{\sch_{\Bcat} \ar[r]^{\widehat{F}_\Z} &\sch_{\Ab}}\,,
\end{equation}
obtained by composing the functor $\widehat{F}\colon \sch_{\Bcat} \to \sch_{\Mon_0}$ in eq.~\ref{Bschemesadjunctionseq} with the functor 
$$\trait\otimes_{\N} \Z\colon \sch_{\Mon_0} \to \sch_{\Ab}$$ defined in \cite[Prop.~3.4]{TV}. Of course, this functor restricts to a functor
\begin{equation}\label{Zbluefunctor}
\xymatrix{\widetilde{\sch_{\Bcat}} \ar[r]^{\widehat{F}_\Z} &\sch_{\Ab}}\,.
\end{equation}

A $\nbcat$-scheme  gives rise, through the functors $\widehat\rho$ and $\widehat{F}_\Z$, to a pair consisting of a monoidal scheme and a classical scheme.
\begin{definition}\label{defBscheme}
Given a $\nbcat$-scheme $\Sigma$, we set
\begin{itemize}
\item $\Sigma_\Z : = \widehat{F}_\Z(\Sigma)$, which is an object of $\sch_\Ab$ (i.e.~a classical scheme); 
\item $\underline{\Sigma}: =\widehat{\rho}(\Sigma)$, which is an object of $\sch_{\Set_\ast}$ (i.e.~a monoidal scheme).
\end{itemize}
\end{definition}
There is a natural transformation  $\Sigma_\Z\to \underline{\Sigma}\adjz$, which is obtained via the unit of the adjunction $\widehat{\rho}\dashv\widehat{\sigma}$  and by applying the functor $\widehat{F}_\Z$.  
By definition, there is indeed a map
\[
\widehat{F}_\Z \Sigma\to \widehat{F}_\Z\widehat{\sigma}\,\widehat{\rho}\, \Sigma\cong\underline{\Sigma}\adjz\,,
\]
where the isomorphism is given by the natural isomorphism $\widehat{F}_\Z\circ\widehat{\sigma}=\trait\adjz$.

In the affine case, such a map is simply realized as the bottom arrow of the map between arrows
\[
\xymatrix{
A\adjz \ar[d] \ar[r] & A\adjz \ar[d] \\
A\adjz \ar[r] & R\otimes_{\N} \Z
}
\]
where the top and the left map are identities.

Summing up, a $\nbcat$-scheme $\Sigma$ induces therefore the following objects:
\begin{align}
\bullet\ &\text{a monoidal scheme $\underline{\Sigma}$;}\label{Bschemesdata1} \\
\bullet\ &\text{a (classical) scheme $\Sigma_\Z$ over $\Z$;}\label{Bschemesdata2}\\
\bullet\ &\text{a natural transformation} \
\Lambda\colon \Sigma_\Z\to\underline{\Sigma}\circ |\trait |\cong\underline{\Sigma}\adjz\,.\label{structuralmapBscheme}
\end{align}
We shall say that the $\nbcat$-scheme $\Sigma$ {\it generates} the pair $(\underline\Sigma, \Sigma_\Z)$, the natural transformation \ref{structuralmapBscheme} being omitted.

\section{An application: $\nbcat$-schemes and $\fun$-schemes}\label{SectionFinal2}
The geometric data \ref{Bschemesdata1}, \ref{Bschemesdata2}, \ref{structuralmapBscheme} appear to be similar to (but different from) those used by A.~Connes and C.~Consani \cite{CC} in their definition of $\fun$-scheme, which is as follows.
\begin{definition}\label{CCschemes}\cite[Def.~4.7]{CC} 
An $\fun$-scheme is a triple $(\underline{\Xi}, \Xi_\Z, \Phi)$, where
\begin{enumerate}
\item $\underline{\Xi}$ is a monoidal scheme;
\item  $\Xi_\Z$ is a (classical) scheme;
\item $\Phi$ is a natural transformation $\underline{\Xi}\to \Xi_\Z\circ (\trait\adjz)$, such that the induced natural transformation 
$\underline{\Xi}\circ \vert\trait\vert \to \Xi_\Z$, when evaluated on fields, gives isomorphisms (of sets).\!\footnote{In \cite{CC} the functor  $\trait\adjz$ is denoted by $\beta$ and its right adjoint $|\trait |$ by 
$\beta^\ast$.} 
\end{enumerate}
\end{definition}

A manifest  difference between $\nbcat$-schemes and $\fun$-schemes is, of course, the direction of the natural transformation linking the monoidal scheme and the classical scheme. Moreover, 
the condition on $\Phi$ in Definition \ref{CCschemes}(3) may fail to be fulfilled in the case of $\nbcat$-schemes, as shown by the following example.
\begin{ex}
Consider a pair $(A,R\to A\adjz)$ defining an affine $\fun$-scheme in the sense Definition \ref{CCschemes}.
Notice that, in this case, the natural transformation $\Phi$ calculated on a field $k$ corresponds to mapping a prime ideal $\mathfrak p$ of $A\adjz$ plus an immersion $A\adjz/{\mathfrak p}\hookrightarrow k$ to their restrictions to $R$; the requirement is that this is a bijection.

On the other hand, according to the general idea underlying the notion of blueprint, if the pair $(A,R)$ is associated with an affine $\Bcat$-scheme (which is, of course, the same thing as an affine $\nbcat$-scheme), then the ring $R$ encodes the information of a relation $\mathcal{R}$ intended to \it reduce\rm\ the number of ideals of $A$. Take for instance the case $(A,A\adjz\to R)$, with $A=\N\cup\{\ -\infty\}$ (additive notation) and $R=A \adjz/(2T-1)$. Then,  $\N$ is an ideal not coming from any ideal of $R$, since  $T$ is invertible (in more algebraic terms, we are saying that the map to any field $k$ sending $T$ to 0 can not be lifted to a map from $R$ to $k$).
\end{ex}

The category $\widetilde{\sch_{\nbcat}}$ and  that of $\fun$-schemes may be combined into a larger category.
\begin{definition}\label{definitionF1schemewr} The category of  ${\fun}$-schemes with relations is the fibered product 
of the category $\widetilde{\sch_{\nbcat}}$ of $\nbcat$-schemes and that of $\fun$-schemes over the category of monoidal schemes. Thus, a $\nbcat$-scheme
$\Sigma$ generating the pair $(\underline\Sigma, \Sigma_\Z)$ and an $\fun$-scheme 
$(\underline{\Sigma}, \Sigma_\Z', \Phi)$ will
determine a ${\fun}$-scheme with relations denoted by the quadruple $(\underline{\Sigma}, \Sigma_\Z, \Sigma_\Z', \Phi)$.
\end{definition}

\noindent
\begin{rmark}\label{finalremark}
Recall that the aim of Lorscheid's definition of blueprint is to increase the amount of closed subschemes of a monoidal scheme. If we loosely refer to the features of the underlying topological space as ``shape'' of the scheme, we could say that  the category of $\Bcat$-schemes (or that of $\nbcat$-schemes) adds ``extra shapes'' to Deitmar's category of monoidal schemes.

Consider now $\fun$-schemes, and let us restrict our attention to the affine case. So, we just have a ring $R$, a monoid $M$, and a map $R\to M\adjz$. Since it is required, by definition, that points remain the same, the monoid is not enriched with ``extra shapes''. However, if we think of the given map as a restriction map between the spaces of functions of the  affine schemes $ M\adjz$ and $R$, we can interpret the datum of the $\fun$-scheme as an enlargement of the space of functions of the affine monoidal scheme $M$.

In conclusion, an $\fun$-scheme with relation, according to the definition \ref{definitionF1schemewr}, allows us both
to add ``extra shapes" to the underlying monoidal scheme and to enlarge its space of functions.

As an example, consider the affine $\fun$-scheme with relation given by the free monoid on four generators and the data
\[
\Z [T_1, T_2, T_3, T_4,\varepsilon ]/(\varepsilon^2)\to\Z [T_1, T_2, T_3, T_4]\to\Z [T_1, T_2, T_3, T_4] / (T_1 T_4 - T_2 T_3 -1)\,.
\]
The $\Bcat$-scheme component on the right has been already taken into consideration in the Introduction; the $\fun$-scheme component on the left adds a nilpotent component to the ring of functions.
\end{rmark}


Notice that the classical scheme $\Sigma_\Z$ is derived from the $\nbcat$-scheme $\Sigma$ via the functor $\widehat{F}_\Z\colon \bluep \to \Ring$ (Definition \ref{defBscheme}). This means, in particular, that the affine $\Bcat$-scheme 
$\Sigma= (M, M\otimes_{\fun} N \to R)$ generates the  affine classical scheme $\Sigma_\Z= R\otimes_\N \Z$. So Definition \ref{definitionF1schemewr} indicates that, as long as we wish to investigate a relationship between this affine $\Bcat$-scheme with an $\fun$-scheme and its associated affine classical scheme $\Sigma'_\Z$,  we are no longer
concerned with the ``monoid relations'' given the map $M\otimes_{\fun} \N\to R$, but  only with the ``ring relations'' given by the map $M\otimes_{\fun} \Z \to R\otimes_\N \Z$ (cf.~eq.~\ref{structuralmapBscheme}). 

From this viewpoint it appears more natural to work with blueprints with ``ring relations''. 
More precisely, 
consider the functor
\[
\trait\adjz: \Mon_0\to \Ring\, 
\]
which is the left adjoint to the forgetful functor, and consider the category $(\trait\adjz)/\Ring$.
We shall denote by  $\Z\trait\bluep$ the full subcategory of $(\trait\adjz)/\Ring$ formally defined in the same way as the subcategory blueprints $\bluep$ of $\trait\adj/\SRing$. Analogously, one defines the category $\Z\trait\nbluep$.
A $\Z\trait\nbcat$-scheme is then a {\it scheme in $\sch_{\Z\trait\nbluep}$  that admits a Zariski  cover by affine schemes in $(\Z\trait\bluep)^{\text{\rm op}}$.}
We shall adopt hereafter the following terminological convention.
\begin{convention}\label{convention} {\it In what follows, by $\nbcat$-scheme we mean a  $\Z\trait\nbcat$-scheme, and by ${\fun}$-scheme with relations we mean the combination  
of a $\Z\trait\nbcat$-scheme and an $\fun$-scheme in the sense of Definition  \ref{definitionF1schemewr}}.
\end{convention}

Now, Definition \ref{defBscheme} and Definition \ref{CCschemes} imply that, for every ${\fun}$-scheme with relations 
$(\underline{\Sigma}, \Sigma_\Z, \Sigma_\Z', \Phi)$,  
there is a natural transformation $\Psi_1 \colon \Sigma_\Z  \to \Sigma_\Z'$ given by the composition
\begin{equation}\label{firsttransferringmap}
\xymatrix{&\Sigma_\Z \ar[r]^{\Lambda\ \ } &\underline{\Sigma}\circ |\trait | \ar[r]^{\  \ \Phi} &\Sigma_\Z'}\,,
\end{equation}
which will be called the  {\em first transferring map} determined by the given $\fun$-scheme with relations.
As its name would suggest, the natural transformation $\Psi_1$, loosely speaking, conveys information on about how many ``points" of  $\Sigma_\Z'$
are compatible with the $\nbcat$-scheme  that generates  the pair $(\underline\Sigma,\Sigma_\Z)$. Actually, there is a different way to ``transfer'' this information from the $\nbcat$-scheme to the $\fun$-scheme associated with the fibered object
$(\underline{\Sigma}, \Sigma_\Z, \Sigma_\Z', \Phi)$. 

The counit of the adjunction $\trait\adjz\dashv |\trait |$ induces a map 
\begin{equation}\label{Bmap1}
\underline{\Sigma}\circ |\trait |\to\underline{\Sigma}\circ ||\trait |\adjz|\,.
\end{equation} 
Moreover, the natural transformation \ref{structuralmapBscheme} induces a map
\begin{equation}\label{Bmap2}
\Sigma'_\Z\circ (|\trait |\adjz)\to\underline{\Sigma}\circ ||\trait |\adjz | \,.
\end{equation}
Let $\Sigma'_{\Bcat}$ be the sheaf on the category $\Ring$ obtained as the pullback of the maps \ref{Bmap1} and \ref{Bmap2}, i.e.
\begin{equation}\label{diagramdefiningSigma'Bcat}
\xymatrix{\Sigma'_{\Bcat} \ar[r] \ar[d] & \Sigma'_\Z\circ (|\trait |\adjz)\ar[d]\\
\underline{\Sigma}\circ |\trait | \ar[r] & \underline{\Sigma}\circ ||\trait |\adjz|
}
\end{equation}
By composing the vertical arrow on the left with $\Phi$, we get 
a  natural transformation 
\begin{equation} \Psi_2\colon \Sigma'_{\Bcat}\to \Sigma'_\Z\,,
\end{equation} 
which will be called the {\em second transferring map} determined by
the $\fun$-scheme $(\underline{\Sigma}, \Sigma_\Z, \Sigma_\Z', \Phi)$.
 
In the case of an $\fun$-scheme $(\underline{\Sigma}, \Sigma_\Z', \Phi)$, the natural transformation $\Phi$ induces an isomorphism $\underline\Sigma(\vert{\mathbb K}\vert) \simeq \Sigma_\Z'(\mathbb K)$ for every field $\mathbb K$.
Since for the finite field $\fq$, one has $\vert\fq\vert = {\mathbb F}_{1^{q-1}}$, it immediately follows, as observed in \cite{CC}, that there is a bijective correspondence between the set of $\fq$-points of  $\Sigma_\Z'$ and the set of 
${\mathbb F}_{1^{q-1}}$-points of $\underline\Sigma$; in others words, one has 
\begin{equation}\label{countingpointsofF1schemes} 
\#\Sigma_\Z'(\fq)= \# \underline\Sigma({\mathbb F}_{1^{q-1}})\,.
\end{equation}
This result can be extended to our setting in two different ways, because, for a $\nbcat$-scheme underlying an $\fun$-scheme with relations, we can think of its ``${\mathbb F}_{1^{q-1}}$-points'' in two different senses.

On the one hand, the forgetful functor
$\vert\trait\vert : \Ring \to \Mon_0$ admits the obvious factorization  
\begin{equation}\xymatrix{\Ring \ar[r]^{G_\Z} & \Z\trait\bluep \ar[r]^\rho &\Mon_0}\,,
\end{equation}
(cf.~eq.~\ref{MonBluepadjuction}).  Clearly, one has 
$$G_\Z(\fq)= (\mathbb{F}_{1^{q-1}}, \mathbb{F}_{1^{q-1}}\adjz \to \mathbb{F}_{1^{q-1}})$$
and $\rho(G_\Z(\fq)) = \vert \fq\vert = \mathbb{F}_{1^{q-1}}$. 
Now, by definition, the first transferring map $\Psi_1$ factorises as $\Psi_1 = \Phi\circ \Lambda$. Since $\Phi$ gives isomorphisms (of sets) when evaluated on fields and 
$\Lambda$ is always locally injective, it is immediate to prove the following result.
\begin{prop}\label{propfirsttransferringmap} Let $(\underline{\Sigma}, \Sigma_\Z, \Sigma_\Z', \Phi)$  be an $\fun$-scheme with relations. The first transferring  map $\Psi_1\colon \Sigma_\Z\to \Sigma_\Z'$,  when evaluated on a field, gives an injective map (of sets). 
In particular, the set of $G_\Z(\fq)$-points of the underlying $\nbcat$-scheme naturally injects into the set of  $\mathbb{F}_q$-points of the scheme $\Sigma''_\Z$ (which is isomorphic to the set of  ${\mathbb F}_{1^{q-1}}$-points of the monoidal scheme $\underline\Sigma$).
\end{prop}

On the other hand, one has the immersion $\sigma\colon \Mon_0 \hookrightarrow \Z\trait\bluep$, with 
$$\sigma(\mathbb{F}_{1^{q-1}}) = (\mathbb{F}_{1^{q-1}}, \mathbb{F}_{1^{q-1}}\adjz 
{\buildrel \text{id} \over \longrightarrow} \mathbb{F}_{1^{q-1}}\adjz) \,.$$
Notice that  $G_\Z(\fq)\neq \sigma(\mathbb{F}_{1^{q-1}})$, while $\vert G_\Z(\fq)\vert = \vert \sigma(\mathbb{F}_{1^{q-1}})\vert = \mathbb{F}_{1^{q-1}}$.
\begin{thm}\label{thmsecondftransferringmap} Let $(\underline{\Sigma}, \Sigma_\Z, \Sigma_\Z', \Phi)$  be an $\fun$-scheme with relations. 
The set of $\sigma(\mathbb{F}_{1^{q-1}})$-points of the underlying $\nbcat$-scheme is in natural bijection with the set of $\mathbb{F}_q$-points of the subpresheaf of $\Sigma'_\Z$ given by the image of $\Psi_2\colon \Sigma'_{\Bcat}\to\Sigma'_\Z$.
\end{thm}
\begin{proof}
Since we can work locally, we assume that the underlying scheme is given by a monoid $M$, a ring $R$, and a map $M\adjz\to R$ satisfying the usual conditions.
An $\mathbb{F}_{1^{q-1}}$-point is given by a commutative square
\[
\xymatrix{
M\adjz \ar[d] \ar[r] & \mathbb{F}_{1^{q-1}}\adjz\cong |\mathbb{F}_q|\adjz \ar[d]^{\mathrm{id}} \\
R \ar[r] & \mathbb{F}_{1^{q-1}}\adjz\cong |\mathbb{F}_q|\adjz
}
\]
such that the arrow on the top is induced by a map $M\to\mathbb{F}_{1^{q-1}}$.\\
The datum of a generic commutative square as above is equivalent to the datum of an $\mathbb{F}_q$-point in $\spec R\circ (|\trait|\adjz)$.\\
The fact that the map on the top has the required property is equivalent to the fact that the image of the point above through the restriction map
\[
\spec R(|\mathbb{F}_q|\adjz)\to \spec (M\adjz)(|\mathbb{F}_q|\adjz)
\]
is in the image of the map
\[
\spec M (|\mathbb{F}_q|)\to \spec (M\adjz)(|\mathbb{F}_q|\adjz)
\]
induced by the functor $\trait\adjz$.
\end{proof}

%
%
%
%

We are now interested in the case where the $\funn$-points of the underlying monoidal scheme $\underline\Sigma$ are counted 
by a polynomial in $n$. Some preliminary definitions and results are in order.

A monoidal scheme $\Sigma$ is said to be {\em noetherian}  if it admits a finite open cover by representable subfunctors $\{\spec(A_i)\}$, with each $A_i$ a noetherian monoid. Recall that, as it is proved in \cite[Theorem 5.10 and 7.8]{Gil}, a monoid is noetherian if and only if it is finitely generated. This immediately implies that, for any prime ideal $\frak p \subset M$, the localized monoid $M_{\frak p}$ is noetherian and the abelian group $M_{\frak p}^\times$ of invertible elements in $M_{\frak p}$ is finitely generated. 
 
\begin{rmark} Notice that, given an $\fun$-scheme $(\underline{\Sigma}, \Sigma_\Z', \Phi)$, the fact that  the monoidal scheme $\underline{\Sigma}$ is noetherian does not entail that the scheme $\Sigma_\Z'$ is noetherian as well. Let us consider, for instance, the affine $\fun$-scheme given by $\Z[X,\varepsilon_i]/(\varepsilon_i^2)\to\Z[X]$, with $i\in\N$. 
The monoidal scheme is noetherian, while the ascending chain of ideals $\ldots\subset (\varepsilon_0,\ldots ,\varepsilon_i)\subset (\varepsilon_0,\ldots ,\varepsilon_{i+1})\subset\ldots$ does not have a maximal element. Observe that, as for 
the points of the classical scheme,  the presence of the $\varepsilon_i$'s is immaterial; hence, one has the required isomorphism $\Z[X](\vert\mathbb K\vert) \simeq \Z[X,\varepsilon_i]/(\varepsilon_i^2)(\mathbb K)$ for any field $\mathbb K$.
\end{rmark}

Let $\widetilde{\underline\Sigma}$ the geometrical realization of the monoidal scheme $\Sigma$. Following Connes-Consani's definition \cite[p.~25]{CC}, we shall say that $\underline\Sigma$ is  {\em torsion-free} if, for any $x\in   
\widetilde{\underline\Sigma}$, the abelian group $\Oc_{\underline\Sigma, x}^\times$ is torsion-free.

\begin{lemma}\label{torsion-freemonoidalschemes} A noetherian monoidal scheme $\underline\Sigma$ is torsion-free if and only if, for any finite group $G$ with $\#G=n$ and for any point $x\in   
\widetilde{\underline\Sigma}$,  the number $\#\Hom (\Oc_{\underline\Sigma, x}^\times, G)$ is polynomial in $n$.
\end{lemma}
\begin{proof} Since $\underline\Sigma$ is noetherian, the abelian  group $\Oc_{\underline\Sigma, x}^\times$ is finitely generated by the remark above. So, if $\underline\Sigma$ is also torsion-free, then  
$\Oc_{\underline\Sigma, x}^\times$ is free of rank ${N(x)}$, and, for any finite group $G$ with $\#G=n$, we have 
$\#\Hom (\Oc_{\underline\Sigma, x}^\times, G)=n^{N(x)}$.

For the converse, suppose there is a point $x$ such that $\Oc_{\underline\Sigma, x}^\times$ is not torsion-free. Being noetherian, $\Oc_{\underline\Sigma, x}^\times$ decomposes as a product $\Z^n\times\prod_{i\in\{ 1,\ldots m\}}\Z_{n_i}$.
For each prime number $p_0$ not dividing any of the $n_1,\ldots, n_m$, 
say $p_0>\hbox{LCM}\,(n_1,\ldots, n_m)$, 
the number of elements of $\Hom(\Oc_{\underline\Sigma, x}^\times,\Z_{p_0})$ is then $p_0^n$. Since there are infinitely many such prime numbers, were $\#\Hom(\Oc_{\underline\Sigma, x}^\times,\Z_p)$ a polynomial in $p$, it would be the polynomial $p^n$. Take now a prime number $p_1$ dividing $n_1$; in that case, the number of elements of $\Hom (\Oc_{\underline\Sigma, x}^\times,\Z_{p_1})$ is greater than $p_1^n$.  In conclusion, $\#\Hom(\Oc_{\underline\Sigma, x}^\times,\Z_p)$ cannot be a polynomial in $p$.
\end{proof}
By Lemma \ref{torsion-freemonoidalschemes}, for each noetherian and torsion-free monoidal scheme $\underline\Sigma$, one can define the polynomial
\begin{equation}\label{monoidalpolynomial} 
P(\underline\Sigma, n) = 
\sum_{x\in  \widetilde{\underline\Sigma}} \#\Hom (\Oc_{\underline\Sigma, x}^\times, \funn)\,.
\end{equation}
The following result is proved in \cite{CC} (Theorem 4.10, (1) and (2)). 
\begin{thm} Let  $(\underline{\Sigma}, \Sigma_\Z', \Phi)$ be an $\fun$-scheme such that the monoidal scheme $\underline{\Sigma}$ is noetherian and torsion-free. Then
\begin{enumerate}
\item $\#\underline\Sigma(\funn) = P(\underline\Sigma, n)$;
\item for each finite field ${\mathbb F}_q$ the cardinality of the set of points of the scheme $\Sigma_\Z'$ that are rational over ${\mathbb F}_q$ is equal to $P(\underline\Sigma, q-1)$.
\end{enumerate}
\end{thm}
Note that the last statement immediately follows from eq.~\ref{countingpointsofF1schemes}, which holds true without any additional assumption on the monoidal scheme.

For each $\nbcat$-scheme $\Sigma$ and each abelian group $G$ (in multiplicative notation, with absorbing element $0$), we denote by
$$\Hom_{\Bcat} (\Oc_{\Sigma, x}^\times, G)$$
the subset of $\Hom(\Oc_{\underline\Sigma, x}^\times, G)$ given by the morphisms satisfying the relations encoded in the blueprint structure of $\Sigma$.
Lemma \ref{torsion-freemonoidalschemes} prompts us to introduce the following definition. 

\begin{definition}\label{definitionBschemenoethtorsionfree}  A $\nbcat$-scheme $\Sigma$ is said to be noetherian if the monoidal scheme $\underline\Sigma$ is noetherian.
 A noetherian $\nbcat$-scheme $\Sigma$ is said to be torsion-free if for any finite group $G$,  the number 
 $\#\Hom_{\Bcat} (\Oc_{\Sigma, x}^\times, G)$  is polynomial in $\# G$.
\end{definition}
\begin{rmark}
While in the case of a noetherian torsion-free monoidal scheme $\underline{\Sigma}$ the polynomial $\#\Hom_{\Bcat} (\Oc_{\underline\Sigma, x}^\times, G)$ is always a monic monomial, this is not always the case for a noetherian torsion-free $\nbcat$-scheme. The next example illustrates this point.
\end{rmark}

\begin{ex}
Consider the affine $\Bcat$-scheme $\Sigma$ given by the free monoid $M=\langle T_1,T_2,T_3,T_4\rangle$ generated by four elements with relations given by the natural projection
\[
\Z[T_1,T_2,T_3,T_4]\to\Z[T_1,T_2,T_3,T_4]/(T_1-T_3+T_2-T_4)\,.
\]
Let  $G$ be a finite group (in multiplicative notation, with absorbing element $0$); we look for maps $f\colon M\to G$ together with compatible maps
\[
\xymatrix{
\Z [T_1,T_2,T_3,T_4] \ar[d] \ar[r] & G\adjz \ar[d]^{\mathrm{id}} \\
\Z [T_1,T_2,T_3,T_4]/(T_1-T_3+T_2-T_4) \ar[r]  & G\adjz
}
\]
Since $G\adjz$ is free, to ensure the compatibility of $f$ with the relation $T_1+T_2=T_3+T_4$ one must have that 
either $f(T_1)=f(T_3)$ and $f(T_2)=f(T_4)$ or $f(T_1)=f(T_4)$ and $f(T_2)=f(T_3)$. There are therefore only 3 possible cases for the polynomial expressing the cardinality of $\Hom_{\Bcat} (\Oc_{\Sigma, x}^\times, G)$:
\begin{itemize}
\item $f(T_1)=f(T_2)= f(T_3)= f(T_4)=0$; in this case the polynomial is the constant polynomial $1$;
\item either $f(T_1)=0$ and $f(T_2)\neq 0$ or $f(T_1)\neq 0$ and $f(T_2)= 0$, each case giving rise to two possible cases; therefore, in each of the four possible cases
the polynomial is $n$; 
\item $f(T_1)\neq 0$ and $f(T_2) \neq 0$; in this case the polynomial is $2n^2-n$ (the term $2n^2$ accounts for $2$  possible free nonzero choices on  $f(T_1)$ and $f(T_2)$, that have to be counted twice since either $f(T_1)=f(T_3)$ or $f(T_1)=f(T_4)$, and the term $-n$ accounts for the case  $f(T_1)=f(T_2)$).
\end{itemize}
\end{ex}

Let   $(\underline{\Sigma}, \Sigma_\Z, \Sigma_\Z', \Phi)$  be an $\fun$-scheme with relations such that
the underlying $\nbcat$-scheme $\Sigma$ is noetherian and torsion-free. We define the polynomial
$$Q(\underline\Sigma, n) = \sum_{x\in \underline\Sigma}\# \Hom_{\Bcat} (\Oc_{\Sigma, x}^\times, \funn)\,.$$

\begin{prop}\label{finalproposition} In the above hypotheses one has the inequality $Q(\underline\Sigma, n) \leq P(\underline\Sigma, n)$.
\end{prop}
\begin{proof} It is clear that 
$\Hom_{\Bcat} (\Oc_{\Sigma, x}^\times, \funn) \subset \Hom (\Oc_{\underline\Sigma, x}^\times, \funn)$, since the first set
contains only the monoid morphisms that are compatible with the blueprint structure locally defined around $x$.
\end{proof}

\begin{rmark}
Connes and Consani developed, in their papers \cite{CCgamma1} and \cite{CCgamma2},  an approach based on $\Gamma$-sets, that generalizes their previous theory of $\fun$-schemes. Since the category of $\Gamma$-sets is endowed with a natural monoidal closed structure, one can apply to that framework the general formalism introduced by To\"en and Vaqui\'e in \cite{TV}. In \cite{CCgamma2}, a notion of scheme is defined; this notion is compared to that arising from \cite{TV}, and the two are shown to be different. The situation is thus analogous to that occurring in the case of blue schemes, as described in \cite{Lor16}. Therefore, it seems worth briefly commenting upon Connes-Consani's construction.

Recall that a $\Gamma$-set is a functor
\[
M(\trait )\colon\Gamma^{\mathrm{op}}\to\Set
\]
from the category $\Gamma^{\mathrm{op}}$ of pointed finite sets (denoted $0^+$, $1^+$, $2^+,\ldots$) to $\Set$. Such a notion was first introduced by Segal \cite{Segal}, who used spaces instead of sets, in order to model commutative monoid and group structures up to homotopy. To this aim, he restricts to the case of ``special $\Gamma$-spaces'', that is, those functors such that the natural map
\[
M(n^+)\to M(1^+)\times\ldots\times M(1^+)
\]
is an equivalence.
This way, the image of $1^+$ corresponds to the object $M:=M(1^+)$ one is interested in, the map $2^+\to 1^+$ extending the final map $2\to 1$ corresponds to an operation (say, additive) 
\[
M\times M\to M\,,
\]
and all the other data correspond to associativity and commutativity conditions. On the other hand, by using the monoidal structure of the category of $\Gamma$-spaces, one can define a second operation (say, multiplicative) on $M$.

Connes-Consani's basic idea is that it is possible to obtain more general structures on $M$ by dropping the ``special'' condition; in particular, one can get a multiplicative monoid structure as above (but without the addition defined in Segal's setting). 
In order to model such a monoid structure, the only relevant information in the $\Gamma$-set has to be the image $M$ of $1^+$; technically, this is implemented by asking the image of an object $n^+$, which is just the  $n$ fold coproduct of $1^+$ with itself, to be just the $n$ fold coproduct of $M$ with itself.  It is shown in \cite{CCgamma2} that, by following a procedure of this kind, the categories $\Mon$ and $\Ring$ embed in the category $\Gamma$-sets, as well as does the category $\mathfrak{MR}$ as defined in \cite{CC}. In the case of rings (realized as objects in the category $\Gamma$-sets), the corresponding schemes, according to Connes-Consani's definition, coincide with the classical ones \cite[Prop.~7.9]{CCgamma2}, whilst this is not the case for the schemes obtained by applying To\"en-Vaqui\'e's formalism \cite[Lemma 8.1]{CCgamma2}.

As a general consideration, we could say that the notion of scheme defined \cite{TV} places greater emphasis on the overall category, while that  defined in \cite{CCgamma2} focuses more on the intrinsic geometric properties of each single object.
\end{rmark}

\appendix

\section{Fibered categories and stacks}

To give the reader a better understanding of To\"en and Vaqui\'e's general construction presented in Section \ref{relativeschemes}, we briefly review some basic facts about fibered categories, pseudo-functors, and stacks, closely following the exposition in \cite{V} (to which the reader is referred for further details). 

Let $\Ccat$ be any category. Roughly speaking, a stack is a sheaf of categories on $\Ccat$ with respect to some Grothendieck topology (recall Definition \ref{topology}).
\begin{definition}[\cite{V}, Def.~3.1]\label{deffcartasianarrow}
Let  $p_\Fcat\colon\Fcat\to\Ccat$ be a functor. An arrow $\phi: \xi \rightarrow \eta$ of $\Fcat$ is {\it Cartesian} with respect to $p_\Fcat$ if, for any arrow $\psi\colon \zeta \rightarrow \eta$ in $\mathcal{F}$ and any arrow $h\colon p_{\Fcat} \zeta \rightarrow p_{\Fcat}\xi$ in $\Ccat$ with $p_{\Fcat}\phi \circ h = p_{\Fcat}\psi$, there exists a unique arrow $\theta \colon\zeta\rightarrow \xi$ with 
$p_{\Fcat}\theta = h$ and $\phi \circ \theta = \psi$, as in the following diagram:
\[\xymatrix{\zeta  \ar[rrr]^{\psi} \ar[rrd]_{\theta} \ar[dd]  &&& \eta \ar[dd]\\
&& \xi\ar[ru]_{\phi} \ar[dd]\\
p_{\mathcal{F}} \zeta\ar'[rr] [rrr] \ar[rrd]_{h} &&& p_{\mathcal{F}}\eta\\
&& p_{\mathcal{F}}\xi \ar[ru] }
\]
Whenever $\xi \rightarrow \eta $ is a Cartesian arrow of $\Fcat$ mapping to an arrow $U\rightarrow V$ of $\Ccat$, we shall also say that $\xi$ is a {\it pullback  of $\eta$ to $U$}.
\end{definition}
\begin{definition}[\cite{V}, Def.~3.5]\label{deffiberedcategory}
A category $\Fcat$ endowed with a functor $p_\Fcat:\Fcat\to\Ccat$ is said to be fibered over $\Ccat$ (with respect to $p_\Fcat$) if, for any map $f\colon U\to V$ in $\Ccat$ and any object $\eta$ in $\Fcat$ such that $p_\Fcat\eta=V$, there exists a Cartesian map $\phi\colon\xi\to\eta$ in $\Fcat$ such that $p_\Fcat\phi=f$.
\end{definition}
\begin{definition}\cite[Def.~3.9]{V}\label{cleavage}
Given a fibered category $p_\Fcat\colon \Fcat\to\Ccat$ over $\Ccat$, a cleavage is a class $K$ of Cartesian maps in $\Fcat$ such that, for each map $f:X\to Y$ in $\Ccat$ and each object $\xi$ in $\Fcat$ over $Y$, there is exactly one map in $K$ over $f$ with codomain $\xi$; when a cleavage is fixed, this unique map will be denoted by $f^\ast_\xi$, or, by a slight abuse of notation, simply by $f^\ast$, if $\xi$ is clear from the context.
\end{definition}
\begin{rmark}
Let $S$ be a set and $\mathrm{SET}$ the set of small sets. 
The assignment of a map of sets $f\colon S\to \text{SET}$ is obviously equivalent to the assignment of the map $p_F\colon F\to S$, with $F=\coprod_{s\in S}f(s)$ and $p_F$ the natural projection. Notice that, for every $s\in S$, one can recover the set $f(s)$ as the fiber $p_F^{-1}\{ s\}$.\\
If we regard a set as a discrete category, then the notion of fibered category introduced in Definition \ref{deffiberedcategory}  can be interpreted as a generalization of the construction above to the categorical framework.\\
When dealing with a functor $\Ccat^{\mathrm{op}}\to{\mathsf {CAT}}$, however, we have not only objects (namely, the categories which are images of objects of $\Ccat$), but also maps between them (namely, the functors which are images of maps of $\Ccat$). So, the fibers on objects of $\Ccat$ with respect to the fibration $p_\Fcat\colon \Fcat\to\Ccat$ have to be connected by maps. This idea is made precise by the notion of a cartesian arrow introduce in Definition \ref{deffcartasianarrow}: the existence of a Cartesian arrow $\phi \colon \xi\to\eta$ amounts to say that $\xi$ is the image of $\eta$ by the functor $\Fcat_{p_\Fcat\eta}\to\Fcat_{p_\Fcat\xi}$ which is image of the map $p_\Fcat\phi\colon p_\Fcat\xi\to p_\Fcat\eta$. Images of maps in $\Ccat$ are defined likewise by imposing the Cartesian condition and using composition rules in $\Fcat$. But there is one more issue to be considered: when we enconde functors data in properties of the category $\Fcat$ (with respect to $p_\Fcat$), we have to bear in mind that categorical properties make sense only up to isomorphisms, and this is reason why, in general, we may expect to recover the original functor only up to equivalences. This fact leads to the following definition.
\end{rmark}
\begin{definition}[\cite{V}, Def.~3.10]
A {\it pseudo-functor} $\Pi\colon \Ccat^\mathrm{op}\to\mathsf{Cat}$ consists of the following data:
\begin{enumerate}
\item[i)] for each object $U$ of $\Ccat$, a category $\Pi U$;
\item[ii)] for each arrow $f\colon U \rightarrow V$, a functor $f^{\ast}\colon \Pi V \rightarrow \Pi U$;
\item[iii)] for each object $U$ of $\Ccat$, an isomorphism $\epsilon _{U}\colon Id ^{\ast}_{U} \simeq id_{\Pi U}$ of functors $\Pi U \rightarrow \Pi U$;
\item[iv)] for each pair of arrows $U  \overset {f}  \rightarrow V \overset{g} \rightarrow W$, an isomorphism
$$\alpha_{f, g} : f^{\ast} g^{\ast} \simeq (gf)^{\ast} \colon \Pi W \rightarrow\Pi U$$
 of functors $\Pi W \rightarrow \Pi U$. 
\end{enumerate}
These data are required to satisfy some natural compatibility conditions which we do not explicitly describe here{\rm (see \cite[p.~47]{V})}.
\end{definition}
It can be proven that the assignment of a fibered categories over a category $\Ccat$ is equivalent, up to isomorphism, to the assignment of a pseudo-functor $\Ccat^\mathrm{op}\to\mathsf{Cat}$. For this reason, in what follows we will tend not to distinguish between a pseudo-functor and the associated fibered category and, given a fibered category $p_\Fcat:\Fcat\to\Ccat$ and an object $X$ of $\Ccat$, we will denote by $F(X)$ the fiber over $X$.

It can also be shown that any pseudo-functor can be strictified, that is, it admits an equivalent functor (\cite[Th.~3.45]{V}). Nonetheless,  it can be convenient to work with pseudo-functors because  many constructions naturally arising in algebraic geometry produce non strict pseudo-functors. The point of view of fibered categories allows one to deal with pseudo-functors by remaining in the usual context of strict functors.

\begin{ex}
Let us consider To\"en and Vaqui\'e's construction, as summarised in section \ref{relativeschemes}, in the particular case of classical schemes. In this case, the category of interest is $\Ring$, regarded as $\aff_{\Ring}^{\mathrm{op}}$, and there is an assignment mapping each ring $A$ to the category $A\trait\Mod$  and each ring morphism $A\to B$ to the functor $\trait\otimes_AB\colon A\trait\Mod\to B\trait\Mod$. Given two consecutive morphims $A\to B\to C$ and an object $M$ of $A-\Mod$, the objects $C\otimes_B(B\otimes_AM)$ and $C\otimes_AM$ are not equal, but only isomorphic. The previous construction provides therefore  a naturally defined pseudo-functor $\Pi\colon \aff_{\Ring}^{\mathrm{op}} \to \mathsf{Cat}$.\\
The pseudo-functor $\Pi$ can be associated to the fibered category over $\Ring^\mathrm{op}$ defined (up to equivalence) in the following way. Let $\Mod$ be the category whose objects are pairs $(A,M)$ with $A$ a ring and $M$ an $A$-module and whose morphisms are pairs of the form  $(f, \lambda)\colon (A, M) \to ( B,N)$, where $f\colon B\to A$ is a ring morphism and 
$\lambda\colon A\otimes_BN\to M$ is a morphism of  $A$-modules. Then the natural projection $\Mod\to\Ring^\mathrm{op}$ is a fibration corresponding to $\Pi$. For each map $f\colon \spec A\to\spec B$ and for each $B$-module $M$, a natural choice of cartesian lifting is given by $(f\colon B\to A, A\otimes_B M=A\otimes_B M)$.
\end{ex}

\begin{ex}\label{exampleYoneda}Let $\Ccat$ be a category closed under fibered products and denote by $\operatorname{Arr}\Ccat$ the category of arrows in $\Ccat$. Let $p_{\operatorname{Arr}\Ccat}\colon \operatorname{Arr}\Ccat\rightarrow \Ccat$ the functor mapping each arrow $X\rightarrow Y$ to its codomain $Y$ and acting in the obvious way on morphisms in $\operatorname{Arr}\Ccat$. The 
$\operatorname{Arr}\Ccat$ is a fibered category, and its associated pseudo-functor maps an object $X$ to the category 
$\Ccat_{/X}$ and a morphism $X\to Y$ to the pullback functor $\trait\times_YX\colon \Ccat_{/Y}\to\Ccat_{/X}$. Similarly, the functor $(p_{\mathrm{Arr}\Ccat})^\mathrm{op}\colon (\operatorname{Arr}\Ccat)^\mathrm{op}\to\Ccat^\mathrm{op}$ is a fibration (independently of the existence of pullbacks), corresponding to the covariant functor $\Ccat\to\mathrm{Cat}$ acting on objects as above and sending a map $f$ to the composition functor $f\circ\trait$.\\
For each object $X$ of $\Ccat$, the category $\Ccat_{/X}$ is naturally fibered over $\Ccat$ through the projection $\Ccat_{/X}\to\Ccat$ given by the domain functor. This fibration is associated, by identifying a set with the corresponding discrete category, to the functor $\Ccat(\trait,X)$, that is, the image of $X$ by the Yoneda embedding. It is thus not surprising that the pseudo-functor $(p_{\mathrm{Arr}\Ccat})^\mathrm{op}$ can be proven to induce an embedding of $\Ccat$ in the 2-category of categories fibered over $\Ccat$.\\
Using this embedding, the Yoneda lemma can be generalized to categories in the following way. Let us recall that the classical Yoneda lemma states that, for every presheaf $F\colon \Ccat^\mathrm{op}\to\Set$ and every object $X$ of $\Ccat$, there is a natural isomorphism $\Hom(\Ccat(\trait ,X),F)\cong F(X)$, which is obtained by sending a map of presheaves to the image of $1_X$. It can be shown that, for every pseudo-functor $p_F\colon F\to\Ccat$ and every object $X$ of $\Ccat$, there is an equivalence of categories $\Hom_\Ccat(\Ccat_{/X},F)\simeq F(X)$. For a proof, see \cite[3.6.2]{V}:  we just point out  that,  analogously to the classical case, also the equivalence $\Hom_\Ccat(\Ccat_{/X},F)\to F(X)$ is defined on objects by mapping a functor to the image of $1_X$, regarded now as an object of $\Ccat_{/X}$ (hence, in the discrete case, this equivalence essentially gives back the Yoneda isomorphism).
\end{ex}
Example  \ref{exampleYoneda} should make clear that there is a rather strict analogy between the theory of presheaves in sets and the theory of presheaves in categories. We now see how the theory of sheaves extends to the case of categories.\\

Let  $\Ccat$ be a category endowed with a Grothendieck topology. 
Recall from Definition \ref{sheaf} that a covering $\mathcal{U}=\{ U_i\to U\}_{i\in I}$ is associated to the subpresheaf $h_\mathcal{U}$ of $h_U=\Ccat(\trait ,U)$ given by the maps that factorise through some element of $\mathcal{U}$. The inclusion map induces a restriction map $\Hom(h_U,F)\to\Hom(h_\mathcal{U},F)$ for each presheaf $F$, and $F$ is said to be i) separated if this map is injective for each covering $\mathcal{U}$; ii) a sheaf if it is bijective for each covering $\mathcal{U}$.

In passing from sets to categories, it is natural to replace $h_U$ with $\Ccat_{/U}$ (see Example \ref{exampleYoneda}) and, accordingly, $h_\mathcal{U}$ with the full subcategory $\Ccat_{/\mathcal{U}}$ of $\Ccat_{/U}$ whose objects are the maps that factorise through some element of $\mathcal{U}$, ``monomorphism'' with ``embedding'' (that is, ``fully faithful functor'') and ``bijection'' with ``equivalence''. So, by regarding $\Ccat_\mathcal{/U}$ as fibered over $\Ccat$ by the composite map $\Ccat_{/\mathcal{U}}\hookrightarrow\Ccat_{/U}\to\Ccat$, we have the following definition.
\begin{definition}
Given a site $\Ccat$, a fibered category $p_\Fcat:\Fcat\to\Ccat$ is said to be
\begin{enumerate}
\item[i)] a prestack if, for any object $U$ of $\Ccat$ and covering $\mathcal{U}$ of $U$, the restriction functor $\Hom_\Ccat(\Ccat_{/U},\Fcat)\to\Hom_\Ccat(\Ccat_{/\mathcal{U}},\Fcat)$ is an embedding;
\item[ii)] a stack if, for any object $U$ of $\Ccat$ and covering $\mathcal{U}$ of $U$, the restriction functor $\Hom_\Ccat(\Ccat_{/U},\Fcat)\to\Hom_\Ccat(\Ccat_{/\mathcal{U}},\Fcat)$ is an equivalence.
\end{enumerate}
\end{definition}
As already observed, the category $\Hom_\Ccat(\Ccat_{/U},\Fcat)$ is equivalent to $\Fcat(U)$. The category $\Hom(\Ccat_{/\mathcal{U}},F)$ also admits an explicit description, in terms of descent data, that can be thought of as glueing data up to isomorphism, and that we now describe.

Given a site $\Ccat$ and a covering $\mathcal{U}=\{ U_i\to U\}_{i\in I}$, we shall write  $U_{i_1,\ldots i_n}$ as a shorthand for $U_{i_1}\times_U\ldots\times_UU_{i_n}$. Notice that, whenever $\{ i_{j_1},\ldots i_{j_k}\}\subset\{ i_1\ldots i_n\}$, there is a natural projection map $p_{i_{j_1},\ldots i_{j,k}}:U_{i_1,\ldots i_n}\to U_{i_{j_1},\ldots i_{j_k}}$.
\begin{definition}\cite[Def.~4.2]{V}
Let $\Ccat$ be a site, $\Fcat$ a category fibered over $\Ccat$, and $\mathcal{U}= \{ U_{i} \rightarrow U\}$ a covering in $\Ccat$. Let be given a cleavage $K$ (see Definition \ref{cleavage}). An {\it object with descent data} ($\{ \xi _{i} \}, \{\phi_{ij}\}$) on $\mathcal{U}$ is a collection of objects $\xi_{i} \in \mathcal{F}(U_{i})$, together with isomorphisms $\phi_{ij}: pr^{\ast}_{1}\xi_{i} \simeq pr^{\ast}_{2}\xi_{j} $ in  $\mathcal{F}(U_{i}\times_{U} U_{j})$ such that the following cocycle conditions is satisfied:
$$ pr^{\ast}_{13}\phi_{ik} = pr^{\ast}_{12}\phi_{ij}\circ pr^{\ast}_{23}\phi_{jk} : pr^{\ast}_{3}\xi_{k}\rightarrow pr^{\ast}_{1}\xi _{i}\quad \text{\rm for any triple of indices $i, j, k$}\,.
$$
\noindent
The isomorphisms $\phi_{ij}$ are called {\it transition isomorphisms} of the object with descent data.\\
An arrow between objects with descent data
 $$\{\alpha_{i}\}: (\{\xi_{i}\}, \{\phi_{ij}\})\rightarrow (\{ \eta_{i}\}, \{\psi_{ij}\})$$
is a collection of arrows $\alpha_{i} : \xi_{i} \rightarrow \eta_{i} $ in $\mathcal{F}(U_{i})$ with the property that  for each pair of indices $i, j$ the diagram 
\begin{equation}\label{diagraminternalHom1} 
\xymatrix{
pr^{\ast}_{2} \xi_{j} \ar[r]^{pr^{\ast}_{2}\alpha_{j}}\ar[d]^ {\phi_{ij}} & pr ^{\ast} _{2}\eta_{j} \ar[d]^{\psi_{ij}}  \\
pr^{\ast}_{1} \xi_{i}\ar[r] _{pr^{\ast}_{1}\alpha_{i}}& pr^{\ast}_{1} \eta_{i}
}
\end{equation} 
commutes.
\end{definition}
\begin{prop}\cite[Prop. 4.5]{V}
Given a site $\Ccat$, a category $\Fcat$ fibered over $\Ccat$, and a cover $\mathcal{U}$ in $\Ccat$, objects with descent data on $\mathcal{U}$ with arrows between them form a category, which is equivalent to $\Hom(\Ccat_{/\mathcal{U}},F)$.
\end{prop}

\vskip2cm
\footnotesize{
\noindent
{\sc Dipartimento di Matematica, Universit\`a di Genova, via Dodecaneso 35, 16146 Genova, Italy}\\
{\it E-mail addresses:} {\tt bartocci@dima.unige.it} (Claudio Bartocci); {\tt gentili@dima.unige.it} (Andrea Gentili)\\
{\sc Universit\'e Paris Diderot -- CNRS Laboratoire SPHERE, UMR 7219, b\^atiment Condorcet, case 7093, 
5, rue Thomas Mann, 75205 Paris cedex 13}\\
{\it E-mail address:} {\tt jean-jacques.szczeciniarz@univ-paris-diderot.fr} (Jean-Jacques Szczeciniarz)}

\end{document}